\documentclass[reqno,12pt]{amsart}%
\usepackage{amsfonts}
\usepackage{amsmath}
\usepackage{amssymb}

\def\bE{{\mathbb{E}}}

\def\bR{{\mathbb{R}}}
\def\bT{{\mathbb{T}}}

\def\cF{{\mathcal{F}}}

\def\cJ{{\mathcal{J}}}

\def\ba{\boldsymbol{\alpha}}
\def\bep{\boldsymbol{\epsilon}}
\def\bbt{\boldsymbol{\beta}}
\def\bg{\boldsymbol{\gamma}}

\newcommand\zm{(\boldsymbol{0})}

\def\Hep{{\mathrm{H}}}

\newcommand\opr[1]{\mathrm{#1}}

\newcommand\mfk[1]{\mathfrak{#1}}

\def\ba{\boldsymbol{\alpha}}
\def\hker{\mathfrak{p}}

\newtheorem{theorem}{Theorem}
\theoremstyle{plain}
\newtheorem{corollary}[theorem]{Corollary}
\newtheorem{definition}[theorem]{Definition}
\newtheorem{proposition}[theorem]{Proposition}

\newtheorem{remark}[theorem]{Remark}

\newcommand\bel[1]{\begin{equation}\label{#1}}
\newcommand\ee{\end{equation}}

\numberwithin{theorem}{section}
\numberwithin{equation}{section}

\begin{document}
\title[Stationary PAM]
{Time-Homogeneous Parabolic Wick-Anderson Model in One Space Dimension:
Regularity of Solution}
\author{H.-J. Kim}
\curraddr[H.-J. Kim]{Department of Mathematics, USC\\
Los Angeles, CA 90089}
\email[H.-J. Kim]{kim701@usc.edu}
\urladdr{}
\author{S. V. Lototsky}
\curraddr[S. V. Lototsky]{Department of Mathematics, USC\\
Los Angeles, CA 90089}
\email[S. V. Lototsky]{lototsky@math.usc.edu}
\urladdr{http://www-bcf.usc.edu/$\sim$lototsky}

 \subjclass[2000]{Primary
60H15; Secondary 35R60, 60H40}

 \keywords{Malliavin Calculus, Space White Noise,
  Wick Product, Wiener Chaos}

\begin{abstract}
Even though  the heat equation with random potential is
a well-studied object, the particular case of time-independent
Gaussian white noise in one space dimension
has yet to receive the attention it deserves. The paper
 investigates  the stochastic heat
equation with space-only Gaussian white noise on a bounded
interval. The main result is that the
space-time regularity of the solution  is the same for  additive  noise
and  for multiplicative noise in the Wick-It\^{o}-Skorokhod interpretation.
\end{abstract}
\maketitle

\today

\section{Introduction}
\label{sec:Intro}

Consider the stochastic heat equation
\bel{eq000}
\frac{\partial u(t,x)}{\partial t} =
\frac{\partial^2 u(t,x)}{\partial x^2} +
u(t,x) \dot{W},
\ee
where  $\dot{W}$ is a Gaussian white noise.
Motivated by various applications in physics,  equation
\eqref{eq000} is often called parabolic Anderson model with
continuous time and space parameters.

If $W=W(t)$ is a Brownian motion in time, then,
with an It\^{o} interpretation,
a change of variables $u(t,x)=v(t,x)\exp(W(t)-(t/2))$ reduces
\eqref{eq000} to the usual heat equation $v_t=v_{xx}$.

If $W=W(t,x)$ is a two-parameter Brownian motion, or Brownian sheet,
then equation \eqref{eq000} has been studied in detail, from one of the original
references \cite[Chapter 3]{Walsh} to a more recent book \cite{DK14}.
 In particular, the It\^{o} interpretation  is the only option; cf. \cite{HP-WZ-SHE}.

If $W=W(x)$ is a Brownian motion in space, then equation \eqref{eq000}
has  two different interpretations:
\begin{enumerate}
\item Wick-It\^{o}-Skorokhod interpretation
\bel{eq000-sp1}
\frac{\partial u(t,x)}{\partial t} =
\frac{\partial^2 u(t,x)}{\partial x^2} +
u(t,x)\diamond \dot{W}(x),
\ee
where $\diamond$ is the Wick product;
 \item Stratonovich interpretation
\bel{eq000-sp2}
\frac{\partial u(t,x)}{\partial t} =
\frac{\partial^2 u(t,x)}{\partial x^2} +
u(t,x)\cdot\dot{W}(x),
\ee
where $u(t,x)\cdot\dot{W}(x)$ is understood  in the point-wise, or path-wise,
sense.
\end{enumerate}

In \cite{Hu02, Hu15}, equation
\eqref{eq000-sp1} is studied on the whole line  as a part of a more
general class of equations. Two  works dealing  specifically with
 \eqref{eq000-sp1} are \cite{UH96}, where the equation is considered on the whole line, and \cite{VStan-WN}, where the Dirichlet boundary value problem
is considered with a slightly more general random potential.

 According to \cite[Theorem 4.1]{UH96}, the solution of \eqref{eq000-sp1}
is almost H\"{o}lder(1/2) in time and space.
By comparison, the solution of \eqref{eq000-sp2}
is almost H\"{o}lder(3/4) in time  and almost H\"{o}lder(1/2)
in space \cite[Theorem 4.12]{Hu15}, whereas for the equation with
additive noise
$$
\frac{\partial u(t,x)}{\partial t} =
\frac{\partial^2 u(t,x)}{\partial x^2} +
 \dot{W}(x), \ t>0,\ x\in \bR, \ u(0,x)=0,
 $$
  the solution is  almost H\"{o}lder(3/4) in time and is almost
H\"{o}lder(3/2) in space, which follows by applying the
Kolmogorov continuity criterion to
$$
u(t,x)=\int_0^t\int_{\bR} \frac{1}{\sqrt{4\pi s}}\,e^{-(x-y)^2/(4s)}dW(y)ds.
$$

The objective of this paper is to establish optimal space-time
regularity of the solution of
\bel{eq:main}
\begin{split}
\frac{\partial u(t,x)}{\partial t} &=
\frac{\partial^2 u(t,x)}{\partial x^2} +
u(t,x)\diamond \dot{W}(x),\ t>0, \ 0<x<\pi,\\
u_x(t,0)&=u_x(t,\pi)=0, \ u(0,x)=u_0(x),
\end{split}
\ee
 and to  define and investigate the corresponding
  fundamental solution.
We show that the solution of \eqref{eq:main} is almost
H\"{o}lder(3/4) in time and is almost
H\"{o}lder(3/2) in space.
As a result, similar to the case of space-time white
noise, solutions of equations driven by either
additive or multiplicative Gaussian white noise
in space have the same regularity, justifying the optimality
claim in connection with \eqref{eq:main}.

Our analysis relies on the chaos expansion of the
solution and  the Kolmogorov continuity criterion.
Section \ref{sec:CSp} provides the necessary background about
chaos expansion and the Wick product. Section \ref{sec:CSol}
introduces the chaos solution of \eqref{eq:main}. Section
\ref{sec:RCS} establishes basic regularity  of the chaos solution as a
random variable and introduces the main tools necessary for the
proof of the main result. Section \ref{sec:AdN} establishes the
benchmark regularity result for the additive-noise version of
\eqref{eq:main}. The main results, namely,  H\"{o}lder
continuity of the chaos solution of \eqref{eq:main}
 in time and space, are in Sections
\ref{sec:Time} and \ref{sec:Space}, respectively. Section \ref{sec:FS}
is about the fundamental chaos solution of \eqref{eq:main}.
 Section \ref{sec:FD} discusses various  generalizations
 of \eqref{eq:main}, including other types of boundary
 conditions.

We  use the following notations:
$$
f_t(t,x)=\frac{\partial f(t,x)}{\partial t}, \ \
 f_x(t,x)=\frac{\partial f(t,x)}{\partial x}, \ \
 f_{xx}(t,x)=\frac{\partial^2 f(t,x)}{\partial x^2};
$$
$$
\bT^n_{s,t}=\left\{(s_1,\ldots,s_n)\in \bR^n:\
s < s_1 < s_{2} < \cdots
< s_n < t\right\}\!,
$$
$0\leq s<t,\ n=1,2,\ldots$;
$$
(g,h)_0=\int_0^{\pi} g(x)h(x)dx,\ \ \|g\|_0=\sqrt{(g,g)_0},\ \
g_k=(g,\mfk{m}_k)_0,
$$
where $\{\mfk{m}_k,\ k\geq 1\}$ is an orthonormal basis in $L_2((0,\pi))$;
$$
dx^n=dx_1dx_2\cdots dx_n.
$$

\section{The Chaos Spaces}
\label{sec:CSp}

Let $(\Omega, \cF, \mathbb{P})$ be a probability space.
A {\tt Gaussian white noise} $\dot{W}$
on $L_2((0,\pi))$ is a collection of  Gaussian random
variables $\dot{W}(h),\ h\in L_2((0,\pi)),$ such that
\begin{equation}
\label{dW0}
\bE \dot{W}(g)=0,\
\bE \Big( \dot{W}(g)\dot{W}(h)\Big)=(g,h)_0.
\end{equation}
For a Banach  space $X$,
denote by $L_p(W;X)$, $1\leq p<\infty$, the collection of
random elements $\eta$ that are measurable with respect to the
sigma-algebra generated by $\dot{W}(h),\  h\in L_2((0,\pi)),$ and
such that $\bE\|\eta\|_X^p<\infty$.

In what follows, we fix the Fourier cosine
basis $\{\mfk{m}_k,\ k\geq 1\}$ in $L_2((0,\pi))$:
\begin{equation}
\label{cos-basis}
\mfk{m}_1(x)=\frac{1}{\sqrt{\pi}},\ \
\mfk{m}_k(x)=\sqrt{\frac{2}{\pi}}\cos(kx),
\end{equation}
 and define
 \bel{xik}
 \xi_k=\dot{W}(\mfk{m}_k).
 \ee
By \eqref{dW0},  $\xi_k,\ k\geq 1,$ are iid standard Gaussian random variables,
and
 $$
 \dot{W}(h)=\sum_{k\geq 1} (\mfk{m}_k, h)_0\, \xi_k.
 $$
 As a result,
 \bel{WN}
 \dot{W}(x)=\sum_{k\geq 1} \mfk{m}_k(x) \xi_k
 \ee
 becomes an alternative notation for $\dot{W}$;
 of course, the series in \eqref{WN} diverges in
  the traditional sense.

 It follows from \eqref{dW0} that
 $W(x)=\dot{W}(\chi_{[0,x]}) $ is a standard
  Brownian motion on $[0,\pi]$, where
   $\chi_{[0,x]}$ is the indicator function of the interval $[0,x]$.

Denote by ${\mathcal{J}}$ the collection of multi-indices $\ba$
with $\ba=(\alpha_{1},\alpha_{2},\ldots)$
so that each $\alpha_{k}$ is a non-negative integer and
$|\ba|:=\sum_{k\geq1}\alpha_{k}<\infty$. For
$\ba,\bbt\in{\mathcal{J}}$, we define
\[
\ba+\bbt=(\alpha_{1}+\beta_{1},\alpha_{2}+\beta_{2},\ldots),\quad
\ba!=\prod_{k\geq1}\alpha_{k}!.
\]
Also,
\begin{itemize}
\item  $\zm$ is the multi-index with all zeroes;
\item  $\bep(i)$ is the multi-index $\ba$ with $\alpha_{i}=1$
 and $\alpha_{j}=0$ for $j\not=i$;
\item
$\ba-\bbt=(\max(\alpha_1-\beta_1,0),
\max(\alpha_2-\beta_2,0),\dots)$;
\item $\ba^-(i)=\ba-\bep(i)$.
\end{itemize}

An alternative way to describe a multi-index $\ba\in \cJ$
with $|\ba|=n>0$ is by its
{\tt characteristic set} $K_{\ba}$, that is, an ordered
$n$-tuple $K_{\ba}=\{k_{1},\ldots,k_{n}\}$,
 where $k_{1}\leq k_{2}\leq\ldots\leq k_{n}$
indicate the locations and the values of the
non-zero elements of $\ba$:
 $k_{1}$ is the index of the first non-zero element of
$\ba,$ followed by $\max\left(  0,\alpha_{k_{1}}-1\right)  $
 of entries with the same value.
 The next entry after that is the index of the second
non-zero element of $\ba$,
followed by $\max\left(  0,\alpha_{k_{2}}-1\right)  $
 of entries with the same value, and so on.
  For example, if $n=7$ and $\ba=(1,0,2,0,0,1,0,3,0,\ldots)$,
then the non-zero elements of
$\ba$ are $\alpha_{1}=1$,
$\alpha_{3}=2$, $\alpha_{6}=1$, $\alpha_{8}=3$, so
that
$K_{\ba}=\{1,3,3,6,8,8,8\}$:
$k_{1}=1,\,k_{2}=k_{3}=3,\,k_{4}=6,
k_{5}=k_{6}=k_{7}=8$.

Define the collection of random variables
 $\Xi=\{\xi_{\ba}, \ \ba \in{\mathcal{J}}\}$ by
\begin{equation*}
\label{eq:basis}
\xi_{\ba} = \prod_{k}
 \left(
 \frac{\Hep_{\alpha_{k}}(\xi_{k})}{\sqrt{\alpha_{k}!}}
  \right),
\end{equation*}
where $\xi_k$ is from \eqref{xik} and
\begin{equation}
\label{eq:hermite}
\Hep_{n}(x) = (-1)^{n} e^{x^{2}/2}\frac{d^{n}}{dx^{n}}%
e^{-x^{2}/2}%
\end{equation}
is the Hermite polynomial of order $n$.
By a theorem of Cameron and Martin \cite{CM},
$\Xi$ is an orthonormal basis in
$L_2(W;X)$ as long as $X$ is a Hilbert space.
Accordingly, in what follows, we always assume that $X$ is a
Hilbert space.

 For $\eta\in L_2(W;X)$, define
$\eta_{\ba}=\bE\big(\eta\xi_{\ba}\big)\in X$. Then
$$
\eta=\sum_{\ba\in\cJ} \eta_{\ba}\xi_{\ba},\
\bE\|\eta\|_X^2=\sum_{\ba\in\cJ}\|\eta_{\ba}\|_X^2.
$$

We will often need spaces other than $L_2(W;X)$:
\begin{itemize}
\item The space
$$
\mathbb{D}^{n}_2(W;X)=\Big\{ \eta=\sum_{\ba\in\cJ}\eta_{\ba}\xi_{\ba}\in
L_2(W;X):
\sum_{\ba\in\cJ}|\ba|^n\;\|\eta_{\ba}\|_X^2
<\infty\Big\},\ n>0;
$$
\item The space
$$
L_{2,q}(W;X)
=\Big\{ \eta=\sum_{\ba\in\cJ}\eta_{\ba}\xi_{\ba}\in
L_2(W;X): \sum_{\ba\in\cJ}q^{|\ba|}
\|\eta_{\ba}\|_X^2<\infty\Big\},\ q>1;
$$
\item The space $L_{2,q}(W;X)$,  $0<q<1$, which is the closure of
$L_2(W;X)$ with respect to the norm
$$
\|\eta\|_{L_{2,q}(X)}=\left(
\sum_{\ba\in\cJ}q^{|\ba|} \|\eta_{\ba}\|_X^2
\right)^{1/2}.
$$
\end{itemize}
It follows that
$$
L_{2,q_1}(W;X)\subset L_{2,q_2}(W;X),\ q_1>q_2,
$$
and,  for every $q>1$,
$$
L_{2,q}(W;X)\subset \bigcap_{n>0}\mathbb{D}^{n}_2(W;X).
$$
It is also known \cite[Section 1.2]{Nualart} that, for
$n=1,2,\ldots,$ the space
$\mathbb{D}^{n}_2(W;X)$ is the domain of $\mathbf{D}^n$,
the $n$-th power of the Malliavin derivative.

Here is another useful property of the spaces $L_{2,q}(W;X)$.

\begin{proposition}
\label{prop:LcInLp}
If $1<p<\infty$, and $q>p-1$, then
$$
L_{2,q}(W;X)\subset L_p(W;X).
$$
\end{proposition}

\begin{proof}
Let $\eta\in L_{2,q}(W;X)$.
The hypercontractivity property of the Ornstein-Uhlenbeck
operator \cite[Theorem 1.4.1]{Nualart} implies\footnote{In fact,
a better reference is the un-numbered equation at the bottom of
page 62 in \cite{Nualart}.}
$$
\left(
\bE \left\|
\sum_{|\ba|=n} \eta_{\ba}\xi_{\ba} \right\|_X^p
\right)^{1/p} \leq
(p-1)^{n/2}\left(
\sum_{|\ba|=n}\|\eta_{\ba}\|_X^2
\right)^{1/2}.
$$
It remains to apply the triangle inequality, followed
by the Cauchy-Schwarz inequality:
$$
\Big(\bE\|\eta\|_X^p\Big)^{1/p}
\leq
\sum_{n=0}^{\infty}
(p-1)^{n/2}\left(\sum_{|\ba|=n}\|\eta_{\ba}\|_X^2\right)^{1/2}
\leq  \left(
\sum_{n=0}^{\infty}
\left(\frac{p-1}{q}\right)^n
\right)^{1/2}
\|\eta\|_{L_{2,q}(W;X)}.
$$
\end{proof}

\begin{definition}
\label{def:WP}
For $\eta \in L_2(W;X)$ and $\zeta\in L_2(W;\bR)$, the
{\tt Wick product} $\eta\diamond\zeta$ is defined by
\bel{eq:def-WP}
\big(\eta\diamond\zeta\big)_{\ba}
=\sum_{\bbt,\bg\in \cJ:\,\bbt+\bg=\ba}
\ \ \ \left(\frac{\ba!}{\bbt!\bg!}\right)^{1/2}
\eta_{\bbt}\,\zeta_{\bg}.
\ee
\end{definition}

To make sense of $\eta_{\bbt}\,\zeta_{\bg}$,
the definition requires at least one of $\eta,\zeta$ to be real-valued.
The normalization in \eqref{eq:def-WP}
ensures that, for every $n,m,k$,
$$
\Hep_n(\xi_k)\diamond \Hep_m(\xi_k)=\Hep_{n+m}(\xi_k),
$$
where $\xi_k$ is one of the
 standard Gaussian random variables \eqref{xik} and
 $\Hep_n$ is the Hermite polynomial \eqref{eq:hermite}.

 \begin{remark}
If $\eta\in L_2\big(W;L_2((0,\pi))\big)$ and
$\eta$ is adapted, that is,
for every $x\in [0,\pi]$, the random
variable  $\eta(x)$ is measurable with respect to the
sigma-algebra generated by
$\dot{W}(\chi_{[0,y]}), \ 0\leq y\leq x$, then,
 by \cite[Proposition 2.5.4 and Theorem 2.5.9]{HOUZ-2},
$$
\int_0^x \eta(x)\diamond \dot{W}(x)dx =
\int_0^x \eta(x)dW(x),
$$
where the right-hand side is
 the It\^{o} integral with respect to the standard Brownian
motion $W(x)=\dot{W}(\chi_{[0,x]})$.
This connection with the It\^{o} integral does not help
when it comes to  equation \eqref{eq:main}$:$
the structure of the heat kernel implies that, for every
$x\in (0,\pi)$,  the solution $u=u(t,x)$ of \eqref{eq:main}
depends on all of the trajectory
of $W(x),\ x\in (0,\pi),$ and therefore is not adapted as a function
of $x$.
\end{remark}

Given a fixed $\ba\in \cJ$, the sum in \eqref{eq:def-WP}
contains finitely many terms, but, in general,
$\sum_{\ba\in \cJ}
\Big\Vert\big(\eta\diamond\zeta\big)_{\ba}\Big\Vert_X^2
= \infty$ so that $\eta\diamond\zeta$ is not square-integrable.

 Here is  a sufficient condition for the Wick
product to be square-integrable.

\begin{proposition}
\label{prop:WP0}
If
\bel{eq:ex-zt}
\zeta=\sum_k b_k\xi_k,\ b_k\in \bR,
\ee
 and
$\sum_k b^2_k<\infty$, then
$\eta \ \ \mapsto \ \ \eta\diamond \zeta$
is a bounded linear operator from $\mathbb{D}^1_2(W;X)$
to $L_2(W;X)$.
\end{proposition}

\begin{proof}
By \eqref{eq:def-WP},
$$
\bE\|\eta\diamond \zeta\|_X^2 =
\sum_{\ba\in\cJ}
\left\Vert \sum_k \sqrt{\alpha_k}\;
b_k \eta_{\ba^-(k)}\right\Vert^2_X.
$$
By the Cauchy-Schwarz inequality,
$$
\left\Vert \sum_k \sqrt{\alpha_k}\,
b_k \eta_{\ba^-(k)}\right\Vert^2_X
\leq
|\ba|\sum_k b_k^2 \|\eta_{\ba^-(k)}\|_X^2.
$$
After summing over all $\ba$ and shifting the summation index,
$$
\bE\|\eta\diamond \zeta\|_X^2 \leq
\Big(\sum_k b_k^2 \Big)
\sum_{\ba\in \cJ}\big(|\ba|+1\big)\|\eta_{\ba}\|_X^2,
$$
concluding the proof.
\end{proof}

Note that, while $\dot{W}(x)$ is of the
form \eqref{eq:ex-zt} (cf. \eqref{WN}),
Proposition \ref{prop:WP0} does not apply: for a typical
value of $x\in [0,\pi]$,  $\sum_k |\mfk{m}_k(x)|^2=+\infty$.
Thus, without either adaptedness of $\eta$ or square-integrability
of $\dot{W}$, an investigation of  the Wick
product $\eta\diamond \dot{W}(x)$ requires  additional constructions.

One approach (cf. \cite{LR-spn})  is to note that if \eqref{eq:ex-zt}
 is a  linear combination of $\xi_k$,
then, by  \eqref{eq:def-WP}, the number
$$
(\eta\diamond \zeta)_{\ba}=
\sum_k \sqrt{\alpha_k}\,b_k\eta_{\ba^-(k)}
$$
is well-defined for every $\ba\in \cJ$ regardless of
whether the series $\sum_k b_k^2$ converges or diverges.
This observation allows an extension of
 the operation $\diamond$  to
spaces much bigger than $L_2(W;X)$ and $L_2(W;\bR)$;
see \cite[Proposition 2.7]{LR-spn}. In particular, both $\dot{W}$ and
$\eta\diamond\dot{W}$, with
\bel{wp-a}
\Big(\eta\diamond\dot{W}\Big)_{\ba}=
\sum_k \sqrt{\alpha_k}\,\mfk{m}_k\eta_{\ba^-(k)},
\ee
 become generalized random elements with
values in $L_2((0,\pi))$.

An alternative approach, which we will pursue in this paper,
 is to consider $\dot{W}$ and
$\eta\diamond\dot{W}$ as usual (square integrable) random elements
with values in a space of generalized functions.

For $\gamma\in \bR$, define the operator
\bel{LambdaOp}
\Lambda^{\gamma}
=\left(I-\frac{\partial^2}{\partial x^2}\right)^{\gamma/2}
\ee
on $L_2((0,\pi))$ by
\bel{Ldg}
\big(\Lambda^{\gamma} f\big)(x)=
\sum_{k=1}^{\infty} \big(1+(k-1)^{2}\big)^{\gamma/2}
 f_k\mfk{m}_k(x),
\ee
where, for a smooth $f$ with compact support in $(0,\pi)$,
$$
f_k=\int_0^{\pi} f(x)\mfk{m}_k(x)dx;
$$
recall that $\{\mfk{m}_k,\ k\geq 1\}$ is the Fourier cosine
basis \eqref{cos-basis} in $L_2((0,\pi))$ so that
$$
\Lambda^2 \mfk{m}_k(x)=\mfk{m}_k(x)+\mfk{m}''_k(x)
=\big(1+(k-1)^2\big)\mfk{m}_k(x).
$$

If $\gamma>1/2$, then, by \eqref{Ldg},
 \bel{Ker-gm}
\big( \Lambda^{-\gamma}f\big)(x)=
\int_0^{\pi} R_{\gamma}(x,y)f(y)dy,
\ee
where
\bel{Ker-gmG}
R_{\gamma}(x,y)=\sum_{k\geq 1}\big(1+(k-1)^2\big)^{-\gamma/2}
\mfk{m}_k(x)\mfk{m}_k(y).
\ee

\begin{definition}
\label{def:SobSp}
 The Sobolev space $H^{\gamma}_2((0,\pi))$ is
$\Lambda^{-\gamma}\Big(L_2((0,\pi))\Big)$.
The norm $\|f\|_{\gamma}$ in the space is defined by
$$
\|f\|_{\gamma}=\|\Lambda^{\gamma}f\|_{0}.
$$
\end{definition}

The next result is a variation on the theme of
Proposition \ref{prop:WP0}.

\begin{theorem}
\label{th:WP-main}
If $\gamma>1/2,$
then $\eta\mapsto \eta\diamond \dot{W}$
is a bounded linear operator from
 $\mathbb{D}^1_2\big(W;L_2((0,\pi))\big)$ to
$L_2\big(W;H^{-\gamma}_2((0,\pi))\big)$.
\end{theorem}

\begin{proof}
By \eqref{wp-a}, if
$$
\eta=\sum_{\ba\in\cJ}\eta_{\ba}\xi_{\ba}
$$
with $\eta_{\ba}\in L_2((0,\pi))$, then
$$
\big(\eta\diamond \dot{W}\big)(x)=
\sum_{\ba\in\cJ}\left(\sum_k
\sqrt{\alpha_k} \mfk{m}_k(x)\eta_{\ba^-(k)}(x)\right)
\xi_{\ba},
$$
so that
$$
\bE\|\eta\diamond \dot{W}\|_{-\gamma}^2
=
\sum_{\ba\in\cJ}\left\Vert
\sum_k \sqrt{\alpha_k}
\Lambda^{-\gamma}\big(\mfk{m}_k\eta_{\ba^-(k)}\big)
\right\Vert_{0}^2.
$$
By the Cauchy-Schwarz inequality,
$$
\left\Vert
\sum_k \sqrt{\alpha_k}
\Lambda^{-\gamma}\big(\mfk{m}_k\eta_{\ba^-(k)}\big)
\right\Vert_{0}^2
\leq
|\ba|\sum_{k}\int_0^{\pi}
\Big(\Lambda^{-\gamma}\big(\mfk{m}_k\eta_{\ba^-(k)}\big)
\Big)^2(x)dx.
$$
After summing over all $\ba$ and shifting the summation
index,
$$
\bE\|\eta\diamond \dot{W}\|_{-\gamma}^2
\leq
\sum_{\ba\in\cJ} \big( |\ba|+1 \big)
\sum_{k}\int_0^{\pi}
\Big(\Lambda^{-\gamma}\big(\mfk{m}_k\eta_{\ba}\big)
\Big)^2(x)dx.
$$
By \eqref{Ker-gm} and Parsevals's equality,
$$
\sum_{k}\int_0^{\pi}
\Big(\Lambda^{-\gamma}\big(\mfk{m}_k\eta_{\ba}\big)
\Big)^2(x)dx=
\int_0^{\pi}\int_0^{\pi}R_{\gamma}^2(x,y)\eta_{\ba}^2(y)dydx,
$$
 and then  \eqref{Ker-gmG} implies
 $$
 \int_0^{\pi}R_{\gamma}^2(x,y)dx = \sum_{k\geq 1}
\big(1+  (k-1)^{2}\big)^{-\gamma} \mfk{m}_k^2(y)
\leq \frac{2}{\pi}\sum_{k\geq 0}\frac{1}{(1+k^2)^{\gamma}},
 $$
 that is,
 $$
 \int_0^{\pi}\int_0^{\pi}R_{\gamma}^2(x,y)\eta_{\ba}^2(y)dydx
 \leq C_{\gamma} \|\eta_{\ba}\|^2_0,\ \ \
 C_{\gamma}=\frac{2}{\pi} \sum_{k\geq 0}\frac{1}{(1+k^2)^{\gamma}}.
 $$
 As a result,
$$
\bE\|\eta\diamond \dot{W}\|_{-\gamma}^2
\leq C_{\gamma}\sum_{\ba\in \cJ}\big(|\ba|+1\big)
\|\eta_{\ba}\|^2_{0},
$$
concluding the proof of Theorem \ref{th:WP-main}.
\end{proof}

\section{The Chaos Solution}
\label{sec:CSol}

Let $(V,H,V')$ be a normal triple of Hilbert spaces,
that is
\begin{itemize}
\item  $V\subset H\subset V^{\prime}$ and the embeddings $V\subset H$ and
$H\subset V^{\prime}$ are dense and continuous;
\item  The space $V^{\prime}$ is dual to $V$ relative to the inner product
in $H;$
\item  There exists a constant $C_H>0$ such that
$\left\vert (u,v)_{H}\right\vert
\leq C_H\left\Vert u\right\Vert _{V}\left\Vert v\right\Vert
_{V^{\prime}}$ for all $u\in V$ and $v\in H.$
\end{itemize}
An abstract homogeneous Wick-It\^{o}-Skorohod
evolution equation in $(V,H,V')$,
driven by the collection $\{\xi_k,\ k\geq 1\}$ of iid standard
Gaussian random variables, is
\bel{eq:gen}
\dot{u}(t)=\opr{A}u(t)+\sum_k \opr{M}_k u(t) \diamond \xi_k,\
t>0,
\ee
where $\opr{A}$ and $\opr{M}_k$ are bounded linear
operators from $V$ to $V'$. Except for Section \ref{sec:FS},
everywhere else in the paper,
the initial condition $u(0)\in H$ is  non-random.

\begin{definition}
\label{def:CS}
The {\tt chaos solution} of
\eqref{eq:gen} is the collection of functions
$\{u_{\ba}=u_{\ba}(t),\ t>0, \ \ba\in \cJ\}$
satisfying the {\tt propagator}
\begin{equation*}
\begin{split}
\dot{u}_{\zm}(t)&=\opr{A}u_{\zm},\
u_{\zm}(0)=u(0),\\
\dot{u}_{\ba}&=
\opr{A}u_{\ba}+\sum_k \sqrt{\alpha_k}
\opr{M}_ku_{\ba^-(k)},\ u_{\ba}(0)=0,\ |\ba|>0.
\end{split}
\end{equation*}
\end{definition}

It is known \cite[Theorem 3.10]{LR-spn} that if
the deterministic equation $\dot{v}=\opr{A}v$ is
well-posed in $(V,H,V')$, then \eqref{eq:gen} has a unique
chaos solution
\begin{equation}
\label{eq:ua-gen}
\begin{split}
u_{\ba}\left(  t\right)  =\frac{1}{\sqrt{\ba!}}
&\sum_{\sigma
\in{\mathcal{P}}_{n}}\int_{0}^{t}
\int_{0}^{s_{n}}\ldots\int_{0}^{s_{2}}\\
& \Phi_{t-s_{n}}{\opr{M}}_{k_{\sigma(n)}}
\cdots\Phi_{s_{2}-s_{1}}{\opr{M}}_{k_{\sigma(1)}}
\Phi_{s_1}u_0 \, ds_{1}\ldots ds_{n},
\end{split}
\end{equation}
where
\begin{itemize}
\item ${\mathcal{P}}_{n}$ is
the permutation group of the set $(1,\ldots, n)$;
\item $K_{\alpha}=\{k_{1},\ldots,k_{n}\}$ is the
characteristic set of $\ba$;
\item $\Phi_{t}$ is the semigroup generated by ${\opr{A}}$:
 $u_{\zm}(t)=\Phi_{t}u_{0}$.
\end{itemize}
Once constructed, the chaos solution does not depend on the
particular choice of the basis in $L_2(W;H)$
\cite[Theorem 3.5]{LR-spn}. In general, though,
$$
\sum_{\ba\in\cJ} \|u_{\ba}(t)\|_H^2 =\infty,
$$
that is, the chaos solution belongs to a space that is bigger than
$L_2(W;H)$; cf. \cite[Remark 3.14]{LR-spn}.

On the one hand, equation \eqref{eq:main}
 is a particular case of \eqref{eq:gen}:
 $\opr{A}f(x)=f''(x)$  with zero Neumann boundary conditions,
 $\opr{M}_kf(x)=\mfk{m}_k(x)f(x)$,
 $H=L_2((0,\pi)),$ $V=H^{1}((0,\pi))$, $V'=H^{-1}((0,\pi))$.
 The corresponding propagator becomes
 \bel{eq:ppgA}
 \begin{split}
 \frac{\partial {u}_{\zm}(t,x)}{\partial t}&=
 \frac{\partial^2{u}_{\zm}(t,x)}{\partial x^2} ,\
u_{\zm}(0,x)=u_0(x),\\
\frac{\partial{u}_{\ba}(t,x)}{\partial t}&=
\frac{\partial^2 u_{\ba}(t,x)}{\partial x^2}
+\sum_k \sqrt{\alpha_k}
\mfk{m}_k(x)u_{\ba^-(k)}(t,x),\ u_{\ba}(0,x)=0,\ |\ba|>0.
\end{split}
 \ee

 Then existence and uniqueness of the chaos solution
 of \eqref{eq:main} are immediate:
 \begin{proposition}
 \label{prop:CS-gen}
 If $u_0\in L_2((0,\pi))$, then equation \eqref{eq:main},
 considered in the normal triple
 $\Big(H^1((0,\pi)), L_2((0,\pi)), H^{-1}((0,\pi))\Big),$
 has a unique chaos solution.
\end{proposition}

\begin{proof} This follows from
 \cite[Theorems 3.10]{LR-spn}.
 \end{proof}

 On the other hand, equation \eqref{eq:main} has two important
 features that are, in general, not present in \eqref{eq:gen}:
 \begin{itemize}
 \item The semigroup $\Phi_t$ has a kernel $\hker(t,x,y)$:
 \bel{eq:kernel}
 \Phi_tf(x)=\int_0^{\pi}\hker(t,x,y)f(y)dy,\ t>0,
 \ee
 where
 \bel{eq:kernel1}
 \hker(t,x,y)=\sum_{k\geq 1} e^{-(k-1)^2t}\mfk{m}_k(x)\mfk{m}_k(y)=
 \frac{1}{\pi}+\frac{2}{\pi}
 \sum_{k=1}^{\infty}e^{-k^2t}\cos(kx)\cos(ky).
 \ee
\item By Parseval's equality,
\bel{eq:Parseval}
\sum_k \left(\int_0^{\pi} f(x)\mfk{m}_k(x)dx\right)^2=
\int_0^{\pi} f^2(x)dx.
\ee
 \end{itemize}

 In fact, the properties of
the chaos solution of \eqref{eq:main} are closely connected with the properties of
the function $\hker(t,x,y)$ from \eqref{eq:kernel1}. Below are some of the properties we will need.
\begin{proposition}
\label{prop:hker}
For $t>0$ and $x,y\in [0,\pi]$,
\begin{align}
\label{eq:hker0}
&0\leq \hker(t,x,y)\leq \frac{\sqrt{t}+1}{\sqrt{t}},\\
\notag
&|\hker_x(t,x,y)|\leq \frac{4}{t},\ \
|\hker_{xx}(t,x,y)|\leq \frac{27}{t^{3/2}},\ \
|\hker_t(t,x,y)|\leq \frac{27}{t^{3/2}}.
\end{align}
\end{proposition}

\begin{proof}
The maximum principle implies $0\leq \hker(t,x,y)$.
To derive  other inequalities, note that, by
integral comparison,
 $$
 \sum_{k\geq 1} e^{-k^2t} \leq
 \int_0^{\infty} e^{-x^2t}dx = \frac{\sqrt{\pi}}{2\sqrt{t}},\ \ t>0,
 $$
 and more generally, for $t>0,\, r\geq 1$,
 \bel{IntComp-g}
 \sum_{k\geq 1} k^{r}e^{-k^2t}\leq
 \left(\frac{r}{2t}\right)^{(r+1)/2}+
 \int_0^{\infty} x^{r} e^{-x^2t} dx
 \leq \frac{(r+1)^{(r+1)}}{t^{(r+1)/2}}.
 \ee
 To complete the proof, we use
 \begin{align*}
&|\hker(t,x,y)|\leq \frac{1}{2}+\frac{2}{\pi}\sum_{k\geq 1} e^{-k^2t},\
|\hker_x(t,x,y)|\leq \sum_{k\geq 1}k e^{-k^2t},\\
&|\hker_{xx}(t,x,y)|\leq \sum_{k\geq 1}k^2 e^{-k^2t},\
|\hker_t(t,x,y)|\leq \sum_{k\geq 1}k^2 e^{-k^2t}.\\
\end{align*}

\end{proof}

 The main consequence of  \eqref{eq:kernel} and
\eqref{eq:Parseval} is
 \begin{proposition}
 \label{prop-WCNormA}
 (1) For $|\ba|=0$,
 \bel{u0-L2}
 \|u_{\zm}(t,\cdot)\|_0\leq \|u_0\|_0,\ \ t>0,
 \ee
 and
\bel{RF-2}
 |u_{\zm}(s,y)|\leq
  C(p,s,t)\|u_0\|_{L_p((0,\pi))},\ 0<s\leq t,\ 0\leq y\leq \pi,
 \ee
 with
 $$
 C(p,s,t)=
 \begin{cases}
 (1+\sqrt{t}\,)s^{-1/2},& {\rm \ if\ } p=1,\\
 \pi^{1/p'} (1+\sqrt{t}\,)s^{-1/2},
 & {\rm \ if\ } 1<p<+\infty,\ p'=\frac{p}{p-1},\\
 1,\ & {\rm \ if\ } p=+\infty.
 \end{cases}
 $$
 In particular,
 \bel{Cpst}
C(p,s,t) \leq  \pi (1+\sqrt{t})s^{-1/2}
\ee
for all $0<s\leq t$ and $1\leq p\leq +\infty$.

 (2) For $|\ba|=n\geq 1,$
 \bel{mod-a-gen-tx}
\begin{split}
&\sum_{|\boldsymbol{\alpha}|=n} |u_{\boldsymbol{\alpha}}(t,x)|^2\\
 &\leq n!\int_{(0,\pi)^n}\left(\int_{\bT^n_{0,t}}
\hker(t-s_n,x,y_n)\cdots \hker(s_{2}-s_1,y_{2},y_1)u_{\zm}(s_1,y_1)
 ds^n\right)^2  dy^n.
 \end{split}
\end{equation}
\end{proposition}

\begin{proof} (1) For $|\ba|=0$,
$$
u_{\zm}(s,y) =  \int_0^{\pi} \hker(s,y,z)u_0(z)dz,
$$
with $\hker$ from \eqref{eq:kernel1}.
Then
$$
\|u_{\zm}(s,\cdot)\|_0=\sum_{k\geq 1} e^{-(k-1)^2s} \ u_{0,k}^2,
$$
from which \eqref{u0-L2} follows.

To derive \eqref{RF-2} when $p<\infty$,
 we use  the H\"{o}lder inequality and \eqref{eq:hker0};
 if $p=+\infty$, then we use $\int_0^{\pi}  \hker(s,y,z) dz=1$ instead of the
 upper bound in \eqref{eq:hker0}.

(2) It follows from \eqref{eq:ua-gen} that, for $|\ba|\geq 1$,
\begin{equation}
\label{eq:ua-gen-1}
\begin{split}
u_{\ba}(t,x)=&\frac{1}{\sqrt{\ba!}}\sum_{\sigma\in \mathcal{P}_n}
\int_{(0,\pi)^n}\int_{\bT^n_{0,t}}
 \hker(t-s_n,x,y_n)\mfk{m}_{k_{\sigma(n)}}(y_n)\\
&\cdots \hker(s_{2}-s_1,y_2,y_1) \mfk{m}_{k_{\sigma(1)}}(y_1)
u_{\zm}(s_1,y_1)\, ds^n\, dy^n.
\end{split}
\end{equation}
Using \eqref{eq:kernel} and notations
\begin{align}
\notag
\mathfrak{e}_{\ba}(y_1,\ldots,y_n)&=\frac{1}{\sqrt{n!\,\ba!}}
\sum_{\sigma\in \mathcal{P}_n}
\mfk{m}_{k_{\sigma(n)}}(y_n)\cdots \mfk{m}_{k_{\sigma(1)}}(y_1),\\
\label{Fn}
F_n(t,x;y_1,\ldots, y_n)&=
\int_{\bT^n_{0,t}}\hker(t-s_n,x,y_n)
\cdots \hker(s_{2}-s_1,y_2,y_1) u_{\zm}(s_1,y_1)\, ds^n,
\end{align}
we re-write \eqref{eq:ua-gen-1} as
\bel{FCAlfa}
u_{\ba}(t,x)=\sqrt{n!}\int_{(0,\pi)^n}
F_n(t,x,y_1,\ldots, y_n)\mathfrak{e}_{\ba}(y_1,\ldots,y_n) dy^n.
\ee
The collection $\{\mathfrak{e}_{\ba}, \ |\ba|=n\}$ is an
orthonormal basis in the symmetric part of the space
$L_2\big((0,\pi)^n\big)$,  so that
 $u_{\ba}$ becomes the
 corresponding Fourier coefficient of the function $F_n$,
 and \eqref{mod-a-gen-tx} becomes  Bessel's inequality.
\end{proof}

\begin{remark}
\label{rem:sym}
 It follows from \eqref{FCAlfa} that
$$
\sum_{|\boldsymbol{\alpha}|=n} |u_{\boldsymbol{\alpha}}(t,x)|^2
=n!\int_{(0,\pi)^n}\widetilde{F}_n^2(t,x;y_1,\ldots,y_n)dy^n,
$$
where
$$
\widetilde{F}_n(t,x;y_1,\ldots,y_n)=
\frac{1}{n!}\sum_{\sigma\in \mathcal{P}_n}
F_n(t,x;y_{\sigma(1)},\ldots,y_{\sigma(n)})
$$
is the symmertrization of $F_n$ from \eqref{Fn}. By the Cauchy-Schwarz
inequality,
$$
\|\widetilde{F}_n\|_{L_2((0,\pi)^n)}\leq
\|F_n\|_{L_2((0,\pi)^n)},
$$
and a separate analysis is necessary to establish a more precise connection
between $\|\widetilde{F}_n\|_{L_2((0,\pi)^n)}$ and
$\|F_n\|_{L_2((0,\pi)^n)}$.
The upper bound
\eqref{mod-a-gen-tx} is enough for the purposes of  this paper.
\end{remark}

 \section{Basic Regularity of the Chaos Solution}
 \label{sec:RCS}

The objective of this section is to
 show that, for each $t>0$,  the chaos solution of \eqref{eq:main} is
 a regular, as opposed to generalized,  random variable, and to  introduce the main
techniques necessary to establish better regularity of the solution.

 \begin{theorem}
 \label{th:Lp}
If $u_0\in L_2((0,\pi))$, then, for every $t>0$, the solution
of \eqref{eq:main} satisfies
 \bel{eq:Lq-int}
 u(t,\cdot)\in \bigcap_{q>1} L_{2,q}\big(W;L_2((0,\pi))\big).
 \ee
 \end{theorem}

 \begin{proof}
It follows from \eqref{mod-a-gen-tx} that
 \bel{mod-a-gen}
 \begin{split}
 &\sum_{|\boldsymbol{\alpha}|=n} |u_{\boldsymbol{\alpha}}(t,x)|^2\\
 &\leq n!\int_{(0,\pi)^n}\int_{\bT^n_{0,t}}\int_{\bT^n_{0,t}}
\Big(
\hker(t-s_n,x,y_n)\cdots \hker(s_{2}-s_1,y_{2},y_1)u_{\zm}(s_1,y_1)
\\
&
\phantom{n!\int_{(0,\pi)^n}\int_{\bT^n_{0,t}}\int_{\bT^n_{0,t}}}
\times
\hker(t-r_n,x,y_n)\cdots \hker(r_{2}-r_1,y_{2},y_1)u_{\zm}(r_1,y_1)
\Big)\,
 ds^n\, dr^n\  dy^n.
 \end{split}
\ee
We now integrate both sides of
\eqref{mod-a-gen} with respect to $x$ and use the semigroup property
 \bel{semigr}
 \int_0^{\pi} \hker(t,x,y)\,\hker(s,y,z)\,dy=
 \hker(t+s,x,z)
 \ee
together with   \eqref{eq:hker0} to evaluate the
integrals over $(0,\pi)$ on the right-hand side,
 starting from the outer-most integral.
We also use \eqref{u0-L2}. The result is
\bel{aux-rf1}
 \begin{split}
 \sum_{|\boldsymbol{\alpha}|=n} \|u_{\boldsymbol{\alpha}}(t,\cdot)\|_0^2
& \leq  n!\,\|u_0\|_{0}^2\, \big(1+\sqrt{t}\big)^{2n}
\int_{\bT^n_{0,t}}\int_{\bT^n_{0,t}}
(2t-s_n-r_n)^{-1/2}\\
&(s_n+r_n-s_{n-1}-r_{n-1})^{-1/2}
\cdots (s_2+r_2-s_1-r_1)^{-1/2}
 ds^n\, dr^n.
 \end{split}
\ee
Next, we use the inequality $4pq\leq (p+q)^2$, $p,q > 0$, to find
\begin{equation}
\label{ineq_1}
(p+q)^{-1/2}\leq p^{-1/4}q^{-1/4},
\end{equation}
so that
\bel{aux-time}
\begin{split}
&\int\limits_{\bT^n_{0,t}}\int\limits_{\bT^n_{0,t}}
(2t-s_n-r_n)^{-1/2}(s_n+r_n-s_{n-1}-r_{n-1})^{-1/2}
\cdots (s_2+r_2-s_1-r_1)^{-1/2}ds^ndr^n\\
& \leq  \left(\int\limits_{\bT^n_{0,t}}(t-s_n)^{-1/4}(s_n-s_{n-1})^{-1/4}
\cdots (s_2-s_1)^{-1/4}ds^n\right)^2
\!\!\!\!=
\left(\frac{\big(\Gamma(3/4)\big)^n}{\Gamma((3/4)n+1)}\right)^2
 t^{3n/2},
\end{split}
\ee
 where  $\Gamma$  is the Gamma function
$$
\Gamma(y)=\int_0^{\infty} t^{y-1}e^{-t}dt.
$$
The last equality in \eqref{aux-time} follows by induction using
\bel{beta}
\int_0^t s^p (t-s)^qds=t^{p+q+1}\,\frac{\Gamma(1+p)\Gamma(1+q)}
{\Gamma(2+p+q)},\ \ \ p,q>-1.
\ee

Combining \eqref{mod-a-gen}, \eqref{aux-rf1}, and \eqref{aux-time},
$$
\sum_{|\boldsymbol{\alpha}|=n} \|u_{\boldsymbol{\alpha}}(t,\cdot)\|_0^2
\leq n!
\left(\frac{\big(\Gamma(3/4)\big)^n}{\Gamma((3/4)n+1)}\right)^2
\,\big(1+\sqrt{t}\big)^{2n}\,t^{3n/2}\,\|u_0\|_{0}^2.
$$
As a consequence of the Stirling formula,
$$
\Gamma(1+p)\geq \sqrt{2\pi p} \, p^p e^{-p}\ \ {\rm and } \ \
n!\leq 2\sqrt{\pi} n^ne^{-n},
$$
meaning that
\bel{mod-a-gen-1-int}
\sum_{|\boldsymbol{\alpha}|=n}
\|u_{\boldsymbol{\alpha}}(t,\cdot)\|_0^2
\leq C^n(t) {n^{-n/2}}\, \|u_0\|_0^2,\ t>0,
\ee
with
$$
C(t)= \big(4/3\big)^{3/2}\,e^{1/2}\,\Gamma^2(3/4)\,
\big(1+\sqrt{t}\big)^{2}\,t^{3/2}.
$$
 Since
$$
\bE\|u(t,\cdot)\|^2_{L_{2,q}(L_2((0,\pi))}=\sum_{n=0}^{\infty}
q^n \sum_{\ba\in \cJ: |\ba|=n}  \|u_{\ba}(t,\cdot)\|^2_{0},
$$
and the series
$$
\sum_{n\geq 1} \frac{C^n}{n^{n/2}}=
\sum_{n\geq 1}\left(\frac{C}{\sqrt{n}}\right)^n
$$
converges for every $C>1$, we get \eqref{eq:Lq-int} and
conclude the proof of Theorem \ref{th:Lp}.

\end{proof}

\begin{corollary}
 If $u_0\in L_2((0,\pi))$, then the chaos solution is an
 $L_2((0,\pi))$-valued random process and, for all $t\geq 0$,
 $$
 \bE\|u(t,\cdot)\|_{0}^p
 <\infty, \ \ 1\leq p<\infty.
 $$
\end{corollary}

\begin{proof}
This follows from \eqref{eq:Lq-int} and Proposition \ref{prop:LcInLp}.
\end{proof}

We will need a slightly more general family of integrals than the one appearing
on the right-hand side of \eqref{aux-time}:
\bel{ab-Int}
\begin{split}
I_1(t;\alpha,\beta)&=\int_0^t (t-s)^{-\alpha}s^{-\beta}ds,\\
I_n(t;\alpha,\beta)&=\int_{\bT_{0,t}^n}(t-s_n)^{-\alpha}\prod_{k=2}^n (s_k-s_{k-1})^{-1/4}
s_1^{-\beta}ds^n, \ \ n=2,3,\ldots,
\end{split}
\ee
for $\alpha\in (0,1),\ \beta\in [0,1)$.
Note that
$$
I_1(t;\alpha,\beta)=\int_0^t (t-s)^{-\alpha}s^{-\beta}ds=
\frac{\Gamma(1-\alpha)\Gamma(1-\beta)}{\Gamma(2-\alpha-\beta)}\, t^{1-\alpha-\beta}
$$
and
$$
I_n(t;\alpha,\beta)=\int_0^t (t-s_n)^{-\alpha}I_{n-1}(s_n;1/4,\beta)ds_n,\ n\geq 1.
$$
By induction and \eqref{beta},
$$
I_n(t;\alpha,\beta)=
\frac{\big(\Gamma(3/4)\big)^{n-1}\Gamma(1-\alpha)\Gamma(1-\beta)}
{\Gamma\big((3n+5-4\alpha-4\beta)/4\big)}\,t^{(3n+1-4\alpha-4\beta)/4},
$$
and then
\bel{ab-Int-1}
n!\,I_n^2(t;\alpha,\beta) \leq C^n(\alpha,\beta,t)n^{-n/2};
\ee
cf. \eqref{mod-a-gen-1-int}.

Next, we show that the chaos solution of \eqref{eq:main} is,
 in fact, a {\tt random field solution}, that is,
  $u(t,x)$ is  well-defined as a
random variable for every $t>0$, $x\in [0,\pi]$.

\begin{theorem}
\label{th:rf}
 If $u_0\in L_{p}((0,\pi))$ for some $1\leq p\leq \infty$, then,
  for every $t>0$ and $x\in [0,\pi]$,
 \bel{eq:Lq}
 u(t,x)\in \bigcap_{q>1} L_{2,q}(W;\bR).
 \ee
\end{theorem}

\begin{proof}
By Proposition \ref{prop-WCNormA},
inequality  \eqref{mod-a-gen} becomes
\bel{mod-a-gen-prf}
 \begin{split}
 &\sum_{|\boldsymbol{\alpha}|=n} |u_{\boldsymbol{\alpha}}(t,x)|^2
 \leq n!\pi^{2}\big(1+\sqrt{t}\big)^2\|u_0\|^2_{L_p((0,\pi))}\\
 &\int\limits_{(0,\pi)^n}\ \ \ \iint\limits_{\bT^n_{0,t}\times \bT^n_{0,t}}
\Big(
\hker(t-s_n,x,y_n)\cdots \hker(s_{2}-s_1,y_{2},y_1)s_1^{-1/2}
\\
&
\phantom{n!\int_{(0,\pi)^n}\int_{\bT^n_{0,t}}\int_{\bT^n_{0,t}}}
\times
\hker(t-r_n,x,y_n)\cdots \hker(r_{2}-r_1,y_{2},y_1)r_1^{-1/2}
\Big)\,
 ds^n\, dr^n\  dy^n.
 \end{split}
\ee
We now use the semigroup property \eqref{semigr}
together with   \eqref{eq:hker0} to evaluate the
integrals over $(0,\pi)$  on the right-hand side of \eqref{mod-a-gen-prf}
starting from the inner-most integral with respect to $y_1$.
The result is
\bel{aux-prf1}
 \begin{split}
 &\sum_{|\boldsymbol{\alpha}|=n} |u_{\boldsymbol{\alpha}}(t,x)|^2
 \leq  n!\, \pi^{2} \big(1+\sqrt{t}\big)^{2(n+1)}\,
 \|u_0\|_{L_p((0,\pi))}^2
\iint\limits_{\bT^n_{0,t}\times \bT^n_{0,t}}
(2t-s_n-r_n)^{-1/2}\\
&(s_n+r_n-s_{n-1}-r_{n-1})^{-1/2}
\cdots (s_2+r_2-s_1-r_1)^{-1/2}s_1^{-1/2}r_1^{-1/2}
 ds^n\, dr^n.
 \end{split}
\ee
Next, similar to \eqref{aux-time}, we use \eqref{ineq_1}
and \eqref{ab-Int}  to compute
\bel{aux-time-prf}
\begin{split}
\iint\limits_{\bT^n_{0,t}\times \bT^n_{0,t}}&
(2t-s_n-r_n)^{-1/2}(s_n+r_n-s_{n-1}-r_{n-1})^{-1/2}\\
&\cdots (s_2+r_2-s_1-r_1)^{-1/2}s_1^{-1/2}r_1^{-1/2}ds^ndr^n\\
& \leq  \left(\int\limits_{\bT^n_{0,t}}(t-s_n)^{-1/4}(s_n-s_{n-1})^{-1/4}
\cdots (s_2-s_1)^{-1/4}s_1^{-1/2}ds^n\right)^2\\
&=
I_n^2(t;1/4,1/2).
\end{split}
\ee
  Combining \eqref{aux-prf1}
 with \eqref{aux-time-prf} and \eqref{ab-Int-1},
 \bel{final-prf}
\sum_{|\boldsymbol{\alpha}|=n}
|u_{\boldsymbol{\alpha}}(t,x)|^2
\leq C^n(t) {n^{-n/2}}\, \|u_0\|_{L_p((0,\pi))}^2,
\ee
for a suitable $C(t)$.
 Then \eqref{final-prf} leads to \eqref{eq:Lq} in the same way as
 \eqref{mod-a-gen-1-int} lead to \eqref{eq:Lq-int}, completing the
 proof of Theorem \ref{th:rf}.

\end{proof}

\begin{corollary}
For every $t>0$, $x\in [0,\pi]$, and $1\leq p<\infty$,
 $$
 \bE|u(t,x)|^p<\infty.
 $$
 \end{corollary}

 \begin{proof}
This  follows from \eqref{eq:Lq} and
Proposition \ref{prop:LcInLp}.
\end{proof}

Finally, we establish a version of the maximum principle for the chaos solution.

\begin{theorem}
\label{th:positivity}
If $u_0(x)\geq 0$ for all $x\in [0,\pi],$
and $u=u(t,x)$ is a random field solution
 of \eqref{eq:main} such that
 $$
 u\in L_2\big(\Omega\times [0,T], L_p((0,\pi))\big),
 $$
 then, with probability one,
$u(t,x)\geq 0$ for all $t\in [0,T]$ and $x\in [0,\pi]$.
\end{theorem}

\begin{proof}
Let $h=h(x)$ be a smooth function with compact support
in $(0,\pi)$ and define
$$
V(t,x;h)=\bE \left(u(t,x)
\exp\left(\dot{W}(h)-\frac{1}{2}\|h\|^2_{L_2(0,\pi)}\right)
\right).
$$
Writing
$h(x)=\sum_{k=1}^{\infty} h_k\mfk{m}_k(x)$
and
$h^{\ba}=\prod_k h_k^{\alpha_k},$
we find
$$
V(t,x;h)=\sum_{\ba\in \cJ}
\frac{h^{\ba}u_{\ba}(t,x)}{\sqrt{\ba!}}.
$$
By  \eqref{eq:ppgA}, the function
$V=V(t,x;h)$ satisfies
$$
\frac{\partial V(t,x;h)}{\partial t} =
\frac{\partial^2 V(t,x;h)}{\partial x^2} +
h(x)V(t,x;h),\ 0<t\leq T, \ x\in (0,\pi),
$$
with $V(0,x;h)=u_0(x)$ and $V_x(t,0;h)=V_x(t,\pi;h)=0$,
 and then the
maximum principle implies
$V(t,x;h)\geq 0$ for  all $t\in [0,T], \ x\in [0,\pi]$.
The conclusion of the theorem now follows, because the
collection of the random variables
$$
\left\{\exp\left(\dot{W}(h)-\frac{1}{2}\|h\|^2_{L_2((0,\pi))}\right),
 \ \ h
 \ {\rm smooth\ with \ compact\  support \ in } \ (0,\pi)\right\}
 $$
is dense in $L_2(W;\bR)$; cf. \cite[Lemma 4.3.2]{Oksendal}.
\end{proof}

\begin{remark}
If $u=u(t,x)$ is continuous in $(t,x)$, then
there exists a single probability-one subset $\Omega'$ of
$\Omega$ such that $u=u(t,x,\omega)>0$ for all
$t\in [0,T]$, $x\in [0,\pi]$, and $\omega\in \Omega'$.
 \end{remark}

\section{Equation With Additive Noise}
\label{sec:AdN}

The objective of this section is to establish the bench-mark
space-time regularity result for  \eqref{eq:main} by considering the
corresponding equation with additive noise:
\bel{eq:add-n}
\begin{split}
U_t&=U_{xx}+\dot{W}(x),\ t>0,\ x\in (0,\pi),\\
U(0,x)&=0,\ U_x(t,0)=U_x(0,\pi)=0.
\end{split}
\ee

By the  variation of parameters formula, the solution of
\eqref{eq:add-n} is
$$
U(t,x)=\int_0^t\int_0^{\pi} \hker(s,x,y)dW(y)ds.
$$
Using \eqref{eq:kernel1},
\begin{align}
\label{eq:add-n-sol}
U(t,x)&=\frac{t}{\pi}\zeta_0+\frac{2}{\pi}
\sum_{k\geq 1} k^{-2}\big(1-e^{-k^2t})\cos(kx)\zeta_k,\\
\label{AddN-Ux}
U_x(t,x)&=-\frac{2}{\pi}
\sum_{k\geq 1} k^{-1}\big(1-e^{-k^2t}\big)\sin(kx)\zeta_k,
\end{align}
where
$$
\zeta_0=W(\pi),\ \ \ \zeta_k=\int_0^{\pi} \cos(kx)dW(x),\ k\geq 1,
$$
are independent Gaussian random variables with zero mean.
In particular, the series on the right-hand sides of \eqref{eq:add-n-sol}
and \eqref{AddN-Ux} converge with probability one for
every $t>0$ and $x\in [0,\pi]$.

Let us now recall the necessary definitions of the H\"older spaces.
For  a function $f=f(x), \ x\in (x_1,x_2)$,
$-\infty<x_1<x_2<+\infty$, we write
$$
f\in \mathcal{C}^{\alpha}((x_1,x_2)), \ 0<\alpha\leq 1,
$$
or, equivalently,  $f$ is H\"older$(\alpha)$, if
$$
\sup_{x,y\in (x_1,x_2),x\not=y} \frac{|f(x)-f(y)|}{|x-y|^{\alpha}}<\infty.
$$
Similarly,
$$
f\in \mathcal{C}^{1+\alpha}((x_1,x_2))
$$
if $f$ is continuously differentiable on $[x_1,x_2]$ and
$$
\sup_{x,y\in (x_1,x_2),x\not=y} \frac{|f'(x)-f'(y)|}{|x-y|^{\alpha}}<\infty.
$$
We also write $f\in \mathcal{C}^{\beta-}((x_1,x_2))$, or
$f$ is almost H\"older$(\beta)$, if
$f\in \mathcal{C}^{\beta-\varepsilon}((x_1,x_2))$ for every
$\varepsilon\in (0,\beta)$.

The main tool for establishing  H\"older regularity of
random processes is  the Kolmogorov continuity criterion:
\begin{theorem}
 \label{th:KCC}
 Let $T$ be a positive real number and $X=X(t)$, a real-valued random process
 on $[0,T]$.  If there exist
numbers $C>0$,  $p>1$, and  $q\geq p$ such that, for all $t,s\in [0,T]$,
\begin{equation*}
\label{mean-cont0}
\bE |X(t)-X(s)|^q\leq C|t-s|^p,
\end{equation*}
then there exists  a modification of $X$ with
 sample trajectories  that are almost H\"{o}lder$\big((p-1)/q\big)$.
 \end{theorem}

 \begin{proof}
 See, for example Karatzas and Shreve \cite[Theorem 2.2.8]{KarShr}.
 \end{proof}

We now apply Theorem \ref{th:KCC} to the solution of equation
\eqref{eq:add-n}.

\begin{theorem}
\label{th:AdN-reg}
The random field $U=U(t,x)$ defined in \eqref{eq:add-n-sol}
satisfies
\begin{align}
\label{AddN-time}
U(\cdot,x)&\in \mathcal{C}^{3/4-}((0,T)),\ x\in [0,\pi],\ T>0;\\
\label{AddN-tx}
U_x(\cdot,x)& \in \mathcal{C}^{1/4-}((0,T)),\ x\in [0,\pi],\ T>0;\\
\label{AddN-x}
U(t,\cdot)& \in \mathcal{C}^{3/2-}((0,\pi)),\ t>0.
\end{align}
\end{theorem}

\begin{proof}
For every $t>0$ and $x,y\in [0,\pi]$, the random variables
$\tilde{U}(t,x)=U(t,x)-\zeta_0t/\pi$ and $U_x(t,x)$
are  Gaussian, so that, by Theorem \ref{th:KCC}, statements
\eqref{AddN-time}, \eqref{AddN-tx}, and \eqref{AddN-x} will
follow from
\begin{align}
\label{AddN-timeE}
\bE|\tilde{U}(t+h,x)-\tilde{U}(t,x)|^2\leq &C(\varepsilon) h^{3/2-\varepsilon},\
\varepsilon\in (0,3/2),\\
\label{AddN-txE}
\bE|U_x(t+h,x)-U_x(t,x)|^2\leq &C(\varepsilon) h^{1/2-\varepsilon},\
\varepsilon\in (0,1/2),\\
\label{AddN-xE}
\bE|U_x(t,x+h)-U_x(t,x)|^2\leq &C(\varepsilon) h^{1-\varepsilon},\
\varepsilon\in (0,1),
\end{align}
respectively, if we use $p=q\delta/2$ with suitable $\delta$ and
sufficiently large $q$.

Using \eqref{eq:add-n-sol} and \eqref{AddN-Ux}, and keeping in
mind that $\zeta_k,\  k\geq 1,$ are iid  Gaussian with  mean zero
and variance $\pi/2$,
\begin{align}
\label{AddN-timeE1}
\bE|\tilde{U}(t+h,x)-\tilde{U}(t,x)|^2=&\frac{2}{\pi}
\sum_{k\geq 1} k^{-4}e^{-2k^2t}(1-e^{-k^2h})^2\cos^2(kx)
;\\
\label{AddN-txE1}
\bE|U_x(t+h,x)-U_x(t,x)|^2=&\frac{2}{\pi}
\sum_{k\geq 1} k^{-2}e^{-2k^2t}(1-e^{-k^2h})^2\cos^2(kx);\\
\label{AddN-xE1}
\bE|U_x(t,x+h)-U_x(t,x)|^2=&\frac{2}{\pi}
\sum_{k\geq 1} k^{-2}(1-e^{-k^2h})^2\big(\sin(k(x+h))-\sin(kx)\big)^2.
\end{align}
We also use
\begin{align}
\label{ineq-exp}
1-e^{-\theta}\leq \theta^{\alpha}, \ 0<\alpha\leq 1,\ \theta>0,\\
\label{ineq-trig}
\ \sin \theta\leq \theta^{\alpha},\ 0<\alpha\leq 1,\ \theta>0.
\end{align}
Then
\begin{itemize}
\item Inequality \eqref{AddN-timeE} follows from
\eqref{IntComp-g}, \eqref{AddN-timeE1},
and \eqref{ineq-exp} by taking $\alpha<3/4$;
\item Inequality \eqref{AddN-txE} follows from
\eqref{IntComp-g}, \eqref{AddN-txE1},
and \eqref{ineq-exp} with $\alpha<1/4$;
\item Inequality \eqref{AddN-xE} follows from
\eqref{IntComp-g}, \eqref{AddN-xE1},
and \eqref{ineq-trig} with $\alpha<1/2$.
\end{itemize}
\end{proof}

\begin{remark}
\label{rem:compensate}
Similar to \cite[Theorem 3.3]{DK14} in the case of space-time
white noise,  equalities \eqref{eq:add-n-sol} and \eqref{AddN-Ux} imply
that, for every $t>0$, the random field $U(t,x)$ is infinitely differentiable
in $t$, and the random field $U_x(t,x)+B(x)$ is infinitely differentiable in
$x$, where
$$
B(x)=\frac{2}{\pi} \sum_{k\geq 1} k^{-1}\zeta_k\sin(kx)
$$
is a Bronwian bridge on $[0,\pi]$.
\end{remark}

\section{Time Regularity of the Chaos Solution}
\label{sec:Time}

The objective of this section is to show that the chaos solution of
\eqref{eq:main} has a modification that is almost H\"older$(3/4)$
in time.  To simplify the presentation,
we will not distinguish different modifications of the solution.

\begin{theorem}
\label{th:TimeReg}
If $u_0\in \mathcal{C}^{3/2}((0,\pi))$, then the
chaos solution of \eqref{eq:main} satisfies
$$
u(\cdot,x)\in \mathcal{C}^{3/4-}\left((0,T)\right)$$
for every $T>0$ and $x\in [0,\pi]$.
\end{theorem}

\begin{proof}  We need to show that, for every $x\in [0,\pi]$,
$h\in(0,1)$, $\varepsilon\in (0,3/4)$, $t\in (0,T)$, and
$p\in (1,+\infty)$,
$$
\Big(\bE|u(t+h,x)-u(t,x)|^p\Big)^{1/p}  \leq C(p,T,\varepsilon) h^{3/4-\varepsilon}.
$$
Then the statement of the theorem will follow Theorem \ref{th:KCC}.

Recall that  $u_{(\bold{0})}(t,x)$ is the solution of
$$
\frac{\partial u_{(\bold{0})}(t,x)}{\partial t}
=\frac{\partial^2u_{(\bold{0})}(t,x)}{\partial x^2},\ u_{(\bold{0})}(0,x)=u_0(x),
$$
with boundary conditions
$$
\frac{\partial u_{(\bold{0})}(t,0)}{\partial x}=
\frac{\partial u_{(\bold{0})}(t,\pi)}{\partial x}=0,
$$
that is,
$$
\frac{\partial u_{(\bold{0})}(t,x)}{\partial t}=(1-\Lambda^2)u_{(\bold{0})}(t,x);
$$
 the operator $\Lambda$ is defined in \eqref{LambdaOp}.
Applying  \cite[Theorem 5.3]{LSU} to equation
$$
U_t(t,x)=(1-\Lambda^2)U(t,x),\ U(0,x)=\Lambda^{-1}u_0(x),
$$
we conclude that, for each $x\in [0,\pi]$,
\bel{det-time-reg}
u_{(\bold{0})}(\cdot,x)\in \mathcal{C}^{3/4}((0,T)).
\ee

For $n\geq 1$ and $h\in (0,1)$, similar to \eqref{mod-a-gen-tx},
\begin{equation}
\label{eq:Timedifference}
\begin{split}
&\sum_{|\boldsymbol{\alpha}|=n}
\left|u_{\boldsymbol{\alpha}}(t+h,x)-u_{\boldsymbol{\alpha}}(t,x)\right|^2\\
\leq &
n!\int_{(0,\pi)^n}
\Bigg(\int_{\mathbb{T}_{0,t+h}^n}
\hker(t+h-s_n,x,y_n)\cdots \hker(s_{2}-s_1,y_{2},y_1)u_{(\bold{0})}(s_1,y_1)
ds^n\\
-&
\int_{\mathbb{T}_{0,t}^n}
\hker(t-s_n,x,y_n)\cdots \hker(s_{2}-s_1,y_{2},y_1)u_{(\bold{0})}(s_1,y_1)
ds^n\Bigg)^2 dy^n.
\end{split}
\end{equation}

We add and subtract
$$
\int_{\mathbb{T}_{0,t}^n} \hker(t+h-s_n,x,y_n)
\cdots \hker(s_2-s_1,y_2,y_1)u_{(\bold{0})}(s_1,y_1)ds^n
$$
inside the square on the right-hand side of \eqref{eq:Timedifference},
and then use $(p+q)^2 \leq 2p^2+2q^2$ to re-write
\eqref{eq:Timedifference} as
\begin{equation}
\label{eq:Timedifferenceineq}
\begin{split}
&\sum_{|\boldsymbol{\alpha}|=n}
\left|u_{\boldsymbol{\alpha}}(t+h,x)-u_{\boldsymbol{\alpha}}(t,x)\right|^2\\
\leq& 2n!\int_{(0,\pi)^n}
\Bigg(\int_t^{t+h}\int_{\mathbb{T}_{0,s_n}^{n-1}}
\hker(t+h-s_n,x,y_n)\cdots \hker(s_{2}-s_1,y_{2},y_1)
u_{(\bold{0})}(s_1,y_1)
ds^n\Bigg)^2dy^n\\
+&
2n!\int_{(0,\pi)^n}\Bigg(\int_{\mathbb{T}_{0,t}^n}
\Big[\hker(t+h-s_n,x,y_n)-\hker(t-s_n,x,y_n)\Big]
\hker(s_n-s_{n-1},y_n,y_{n-1})\\
&
\cdots \hker(s_{2}-s_1,y_{2},y_1)u_{(\bold{0})}(s_1,y_1)
ds^n\Bigg)^2 dy^n.
\end{split}
\end{equation}
To estimate the first term on the right-hand side of \eqref{eq:Timedifferenceineq},
 we follow  computations
similar to \eqref{aux-time} and \eqref{aux-prf1}, and
use
\bel{time-reg-aux000}
\hker(t,x,y)\geq 0,\ \
\|u_{\mathbf{(0)}}(s,\cdot)\|_ {L_{\infty}((0,\pi))} \leq \|u_0\|_{L_{\infty}((0,\pi))},\ \
s\geq 0,
\ee
as well as \eqref{ab-Int-1}:
\begin{equation}
\label{Time-FirstTerm}
\begin{split}
&2n!\int\limits_{(0,\pi)^n}
\Bigg(\int\limits_t^{t+h}\int\limits_{\mathbb{T}_{0,s_n}^{n-1}}
\hker(t+h-s_n,x,y_n)\cdots \hker(s_{2}-s_1,y_{2},y_1)
u_{(\bold{0})}(s_1,y_1)
ds^n\Bigg)^2dy^n\\
\leq& 2n!\|u_{0}\|_{L_{\infty}((0,\pi))}^2
\Bigg(\int\limits_t^{t+h}\!\int\limits_{\mathbb{T}_{0,s_n}^{n-1}}\!\!
(t+h-s_n)^{-1/4} (s_n-s_{n-1})^{1/4}\cdots (s_2-s_1)^{-1/4}
ds^n\Bigg)^2\!\!dy^n\\
\leq& 2n!\|u_{0}\|_{L_{\infty}((0,\pi))}^2
\Bigg(
\int\limits_t^{t+h}(t+h-s_n)^{-1/4} I_{n-1}(s_n;1/4,1/4)ds_n
\Bigg)^2\\
&\leq
 \|u_{0}\|_{L_{\infty}((0,\pi))}^2 \, C^n(t)  n^{-n/2} \, h^{3/2},
\end{split}
\end{equation}
with a suitable  $C(t)$.

To estimate the second term on the right-hand side of \eqref{eq:Timedifferenceineq},
define
\begin{equation*}
\begin{split}
\mathcal{I}(t,h,s,r,x)=\int_0^{\pi}& \Big(\hker(t+h-s,x,y)-\hker(t-s,x,y)\Big)\\
&\times \Big(\hker(t+h-r,x,y)-\hker(t-r,x,y)\Big)dy.
\end{split}
\end{equation*}

By \eqref{eq:kernel1},
$$
\mathcal{I}(t,h,s,r,x)
=\frac{4}{\pi^2}\sum_{k\geq 1} (e^{-k^2h}-1)^2
e^{-k^2(t-s)-k^2(t-r)}\cos^2(kx).
$$
Using \eqref{ineq-exp} and taking
$0<\gamma<3/4$, we conclude that
$$
\mathcal{I}(t,h,s,r,x)\leq
h^{2\gamma}\sum_{k\geq 1} k^{4\gamma}e^{-k^2(2t-s-r)}.
$$
Then \eqref{IntComp-g} implies
$$
\mathcal{I}(t,h,s,r,x)\leq
(4\gamma+1)^{4\gamma+1}\, h^{2\gamma}\,(2t-s-r)^{-(1/2+2\gamma)}.
$$

Note that  $1/2+2\gamma<2$.

We now carry out  computations similar to  \eqref{aux-prf1}, and
use \eqref{ab-Int-1} and  \eqref{time-reg-aux000}:
\begin{equation}
\label{eq:2ndTerm}
\begin{split}
&2n!\int\limits_{(0,\pi)^n}\Bigg(\ \int\limits_{\mathbb{T}_{0,t}^n}
\Big[\hker(t+h-s_n,x,y_n)-\hker(t-s_n,x,y_n)\Big]\\
&
\times \hker(s_{2}-s_1,y_{2},y_1)
\cdots \hker(s_{2}-s_1,y_{2},y_1)u_{(\bold{0})}(s_1,y_1)
ds^n\Bigg)^2 dy^n\\
&\leq
2n!\|u_{0}\|_{L_{\infty}((0,\pi))}^2 (1+\sqrt{t})^{2n}\\
&\iint\limits_{\mathbb{T}^{n}_{0,t}\times\mathbb{T}^{n}_{0,t}}
\!\!\!\mathcal{I}(t,h,s_n,r_n,x)
\prod^{n-2}_{k=1}(s_{k+1}+r_{k+1}-s_k-r_k)^{-1/2}
(s_{1}+r_{1}-2s)^{-1/2}ds^ndr^n\\
&\leq h^{2\gamma} \|u_{0}\|_{L_{\infty}((0,\pi))}^2
\frac{4(4\gamma+1)^{4\gamma+1}}{\pi}\, (1+\sqrt{t})^{2n}\,
n!\,I_n^2\big(t;1/4+\gamma, 1/4\big)\\
&\leq
\|u_{0}\|_{L_{\infty}((0,\pi))}^2 C^n(t) n^{-n/2} \, h^{2\gamma},
\end{split}
\end{equation}
with a suitable  $C(t)$.

 Combining \eqref{Time-FirstTerm} and \eqref{eq:2ndTerm},
\bel{t-incr}
\sum_{|\boldsymbol{\alpha}|=n}
\left|u_{\boldsymbol{\alpha}}(t+h,x)-
u_{\boldsymbol{\alpha}}(t,x)\right|^2\leq h^{3/2-2\varepsilon}\,
\|u_{0}\|_{L_{\infty}((0,\pi))}^2C^n(t,\varepsilon)n^{-n/2},
\ee
$\varepsilon \in (0,3/4),\ n\geq 1,$ and then, by \eqref{det-time-reg} and Proposition \ref{prop:LcInLp},
\begin{equation*}
\begin{split}
\Big(\mathbb{E}\left|u(t+h,x)-u(t,x)\right|^p\Big)^{1/p}
\leq&\sum_{n=0}^{\infty}(p-1)^{n/2}\left(\sum_{|\boldsymbol{\alpha}|=n}
\left|u_{\boldsymbol{\alpha}}(t+h,x)-
u_{\boldsymbol{\alpha}}(t,x)\right|^2\right)^{1/2}\\
\leq& C(p,T,\varepsilon)\|u_{0}\|_{L_{\infty}((0,\pi))}\,h^{3/4-\varepsilon},
\end{split}
\end{equation*}
for all $1< p<+\infty,\ t\in (0,T),\ \varepsilon \in (0,3/4),\ 0<h<1$,
 completing the proof of Theorem \ref{th:TimeReg}.

\end{proof}

\section{Space Regularity of the Chaos Solution}
\label{sec:Space}

The objective of this section is to show that, for every $t>0$,
 the chaos solution
$$
u(t,x)=\sum_{\ba\in \mathcal{J}} u_{\ba}(t,x)\xi_{\ba}
$$
of \eqref{eq:main} has a modification that is
in $\mathcal{C}^{3/2-}((0,\pi))$.
As in the previous section, we will not distinguish between different modifications of $u$.

To streamline the presentation, we will break the argument in two parts:
existence  of $u_x$ as a random field, followed by H\"{o}lder$(1/2-)$
regularity of $u_x$ in space.

Define
$$
v(t,x)=\sum_{\ba\in \mathcal{J}} v_{\ba}(t,x)\xi_{\ba},
$$
where
\begin{equation*}
\label{eq:va-gen}
\begin{split}
v_{\ba}\left(t,x\right) &=\frac{\partial u_{\ba}(t,x)}{\partial x}\\
&=
\frac{1}{\sqrt{\ba!}}\int_{(0,\pi)^n}
\sum_{\sigma
\in{\mathcal{P}}_{n}}\int_{0}^{t}
\int_{0}^{s_{n}}\ldots\int_{0}^{s_{2}}
 \hker_x(t-s_{n},x,y_n)\mfk{m}_{k_\sigma(n)}(y_n)\\
\times &\hker(s_{n}-s_{n-1},y_n,y_{n-1})\mfk{m}_{k_\sigma(n-1)}
\cdots\hker(s_{2}-s_{1},y_2,y_1)\mfk{m}_{k_\sigma(1)}(y_1)
u_0(s_1,y_1) \, ds^ndy^n.
\end{split}
\end{equation*}

\begin{theorem}
Assume that $u_0 \in L_p((0,\pi))$ for some $1\leq p\leq \infty$.
Then, for every $t>0$ and $x\in (0,\pi) $,
$$
u_x(t,x) \in \bigcap_{q>1} L_{2,q}(W;\mathbb{R}).
$$
\end{theorem}

\begin{proof}
By construction, $v=u_x$ as generalized processes. It remains to
show that
\bel{vx-reg}
v(t,x) \in \bigcap_{q>1} L_{2,q}(W;\mathbb{R}).
\ee
Similar to \eqref{mod-a-gen-tx},
\begin{equation}
\label{eq:vxn}
\begin{split}
\sum_{|\ba|=n} |v_{\ba}(t,x)|^2
\leq & n!\int\limits_{(0,\pi)^n}\Bigg(\int\limits_{\mathbb{T}_{0,t}^n} \hker_x(t-s_n,x,y_n)\hker(s_n-s_{n-1},y_n,y_{n-1})\\
\cdots &\hker(s_{2}-s_1,y_{2},y_1)u_{(\bold{0})}(s_1,y_1)ds^n\Bigg)^2dy^n.
\end{split}
\end{equation}
Using \eqref{Cpst},
\begin{equation*}
\begin{split}
&\sum_{|\ba|=n} |v_{\ba}(t,x)|^2\\
\leq& n!\|u_{0}\|_{L_{p}((0,\pi))}^2\pi^2(1+\sqrt{t})^2
\int\limits_{(0,\pi)^n}\!\!\Bigg(\ \int\limits_{\mathbb{T}^{n}_{0,t}}
\hker_x(t-s_{n},x,y_n)\hker(s_{n}-s_{n-1},y_{n},y_{n-1})\\
&\ \ \ \ \ \
\cdots \hker(s_{2}-s_1,y_{2},y_1)s_1^{-1/2}ds^{n}\Bigg)^2dy^n.
\end{split}
\end{equation*}

By \eqref{eq:kernel1},
$$
\int_0^{\pi}\hker_x(t,x,y)\hker_x(s,x,y)dy
=\frac{4}{\pi^2}\sum_{k=1}^{\infty} k^2 e^{-k^2(t+s)}\sin^2(kx)
\leq \frac{27}{(t+s)^{3/2}},
$$
and then
\begin{eqnarray*}
&&\sum_{|\ba|=n} |v_{\ba}(t,x)|^2
\leq 27n! \pi^{2}(1+\sqrt{t})^{2n}
\|u_{0}\|_{L_p((0,\pi))}^2\\
&\times
& \left(\int_{\mathbb{T}^{n}_{0,t}} (t-s_{n-1})^{-3/4}(s_{n-1}-s_{n-2})^{-1/4}
\cdots (s_2-s_1)^{-1/4}s_1^{-1/2}ds^{n}\right)^2ds\\
&=&27\pi^{2}(1+\sqrt{t})^{2n}\|u_{0}\|_{L_p((0,\pi))}^2\
n!\,I_n^2\big(t;3/4,1/2\big)
\leq \|u_{0}\|_{L_p((0,\pi))}^2 C^n(t)n^{-n/2}
\end{eqnarray*}
with a suitable $C(t)$; cf. \eqref{ab-Int-1}. Then \eqref{vx-reg} follows
in the same way as \eqref{eq:Lq-int} followed from
 \eqref{mod-a-gen-1-int}.

\end{proof}

\begin{remark}Similar to the proof of
 Theorem \ref{th:TimeReg}, an interested reader can confirm that
 $u_x(\cdot,x)\in \mathcal{C}^{1/4-}([\delta,T])$ for every $x\in [0,\pi]$
 and $T>\delta>0$.
\end{remark}

\begin{theorem}
\label{th:SpaceReg}
If $u_0\in L_p((0,\pi))$ for some $1\leq p\leq \infty$,
then, for every $t>0$,
$$
u_x(t,\cdot)\in \mathcal{C}^{1/2-}((0,\pi)).
$$
\end{theorem}

\begin{proof}
We continue to use the notation $v=u_x$.
Then the objective is to show that, for every sufficiently small $h>0$ and
every $x\in (0,\pi)$, $t>0$,  $p>1$, and $\gamma \in (0,1/2),$
\bel{KC-vincr}
\Big(\,\bE |v(t,x+h)-v(t,x)|^p \Big)^{1/p}
\leq C(t,p,\gamma) h^{\gamma};
\ee
then the conclusion of the theorem will follow from the Kolmogorov continuity
criterion.

Similar to \eqref{mod-a-gen-tx},
\begin{equation*}
\begin{split}
&\sum_{|\ba|=n}
\left|v_{\ba}(t,x+h)
-v_{\ba}(t,x)\right|^2\\
\leq &
n!\int_{(0,\pi)^n}
\Bigg(\int_{\mathbb{T}^n_{0,t}}
\left[\hker_x(t-s_n,x+h,y_n)
-\hker_x(t-s_n,x,y_n)\right]\\
\times& \hker(s_{n}-s_{n-1},y_{n},y_{n-1}) \cdots
\hker(s_{2}-s_1,y_{2},y_1)u_{(\bold{0})}(s_1,y_1)
ds^n\Bigg)^2 dy^n,
\end{split}
\end{equation*}
and then
\begin{equation}
\label{ux-incr}
\begin{split}
&\sum_{|\ba|=n}
\left|v_{\ba}(t,x+h)-v_{\ba}(t,x)\right|^2\\
\leq& n!\pi^2\big(1+\sqrt{t}\,\big)^2\|u_0\|_{L_p((0,\pi))}^2
\int\limits_{(0,\pi)^n}\!\!
\Bigg(\,\int\limits_{\mathbb{T}^{n}_{0,t}}
\left[\hker_x(t-s_{n-1},x+h,y_n)
-\hker_x(t-s_{n-1},x,y_n)\right]\\
\times& \hker(s_{n-1}-s_{n-2},y_{n},y_{n-1})
\cdots \hker(s_{1}-s,y_{2},y_1)s_1^{-1/2}
ds^{n}\Bigg)^2dy^n;
\end{split}
\end{equation}
cf.  \eqref{eq:vxn}.

Next, define
\begin{equation*}
\begin{split}
&J(t,s,r,x,y,h)\\
&=
\int_0^{\pi} \left(\hker_x(t-s,x+h,y)
-\hker_x(t-s,x,y)\right)
\left(\hker_x(t-r,x+h,y)
-\hker_x(t-r,x,y)\right)dy.
\end{split}
\end{equation*}
From \eqref{eq:kernel1},
\begin{equation*}
\begin{split}
J(t,s,r,x,y,h)
=\frac{2}{\pi}
\sum_{k\geq 1} k^2e^{-k^2(2t-s-r)}
\big(\cos(k(x+h))-\cos(kx)\big)^2.
\end{split}
\end{equation*}

Using
$$
\cos \varphi - \cos \psi = -2\sin((\varphi-\psi)/2)\, \sin ((\varphi+\psi)/2)
$$
and \eqref{ineq-trig},
and taking $\gamma\in (0,1/2)$,
$$
J(t,s,r,x,y,h)\leq 2h^{2\gamma} (2t-s-r)^{-3/2-\gamma}.
$$

Note that
$$
3/2+\gamma<2.
$$

After expanding the square and using  the semigroup property,
\eqref{ux-incr} becomes
\begin{equation}
\label{sp-incr}
\begin{split}
&\sum_{|\ba|=n}
\left|v_{\ba}(t,x+h)
-v_{\ba}(t,x)\right|^2\\
&\leq 2 h^{2\gamma}n!
\pi^2\big(1+\sqrt{t}\,\big)^{2n}\|u_0\|_{L_p((0,\pi))}^2\\
&\iint\limits_{\mathbb{T}^{n}_{0,t}\times \mathbb{T}^{n}_{0,t}}
(2t-s_{n-1}-r_{n-1})^{-3/2-\gamma}
\prod^{n-2}_{k=1}(s_{k+1}+r_{k+1}-s_k-r_k)^{-1/2}
s_{1}^{-1/2}r_1^{-1/2}ds^{n}dr^{n}\\
&\leq 2 h^{2\gamma}\
\pi^2\big(1+\sqrt{t}\,\big)^{2n}\|u_0\|_{L_p((0,\pi))}^2\
n!\,I_n^2\big(t;3/4+(\gamma/2),1/2\big) \leq C^n(t,\gamma)n^{-n/2};
\end{split}
\end{equation}
cf. \eqref{aux-rf1} and \eqref{ab-Int-1}. Then Proposition \ref{prop:LcInLp}
implies \eqref{KC-vincr}, completing the proof of Theorem \ref{th:SpaceReg}.

\end{proof}

\section{The Fundamental Chaos Solution}
\label{sec:FS}

\begin{definition}
The fundamental chaos solution of \eqref{eq:main} is the
collection of functions
$$
\{\mathfrak{P}_{\ba}(t,x,y),\ t>0, \ x,y\in [0,\pi],\ \ba\in \cJ\}
$$
defined by
\begin{equation}
\label{eq:ua-fund-1}
\begin{split}
\mathfrak{P}_{\zm}(t,x,y)&=\hker(t,x,y),\\
\mathfrak{P}_{\ba}(t,x,y)&=\frac{1}{\sqrt{\ba!}}\sum_{\sigma\in \mathcal{P}_n}
\int_{(0,\pi)^n}\int_{\bT^n_{0,t}}
 \hker(t-s_n,x,y_n)\mfk{m}_{k_{\sigma(n)}}(y_n)\cdots\\
&\cdots \hker(s_{2}-s_1,y_2,y_1) \mfk{m}_{k_{\sigma(1)}}(y_1)
\hker(s_1,y_1,y)\, ds^n\, dy^n.
\end{split}
\end{equation}
\end{definition}

The intuition behind this definition is that \eqref{eq:ua-fund-1} is
the chaos solution of \eqref{eq:main} with initial condition $u_0(x)=\delta(x-y)$.
More precisely, it follows from \eqref{eq:ua-gen-1} that if
\bel{eq:FS-Main}
\mathfrak{P}(t,x,y)=\sum_{\ba\in \cJ} \mathfrak{P}_{\ba}(t,x,y)
\xi_{\ba},
\ee
then
\bel{FCS-det}
u(t,x)=\int_0^{\pi}
\mathfrak{P}(t,x,y)u_{0}(y)dy
\ee
is the chaos solution of \eqref{eq:main} with non-random initial condition
$u(0,x)=u_0(x)$. Before
developing these ideas  any further, let us apply the results of
Sections \ref{sec:CSol}--\ref{sec:Space}
 to the random function $\mathfrak{P}$.

\begin{theorem}
\label{prop:FS}
 The function $\mathfrak{P}$ defined by \eqref{eq:FS-Main} has
the following properties:
\begin{align}
\label{FS-reg1}
&\mathfrak{P}(t,x,y) \in \bigcap_{q>1} L_{2,q}(W;\mathbb{R}),\
t>0,\ \ {\rm \ uniformly\ in \ } x,y\in [0,\pi];\\
\label{FS-sym-pos}
&\mathfrak{P}(t,x,y)\geq 0,\ \ \mathfrak{P}(t,x,y)=\mathfrak{P}(t,y,x),
\ \ \ t>0,\ x,y\in [0,\pi];\\
\label{FS-reg2}
& \mathfrak{P}(\cdot,x,y) \in \mathcal{C}^{3/4-}((\delta, T)),
 0<\delta<T,\ \ x,y\in [0,\pi];\\
 \label{FS-reg3}
&\mathfrak{P}(t,\cdot;y) \in \mathcal{C}^{3/2-}((0,\pi)),\
t>0,\ y\in [0,\pi].
\end{align}
\end{theorem}

\begin{proof}  Using \eqref{eq:hker0}, \eqref{mod-a-gen},
 \eqref{aux-time}, and \eqref{ab-Int-1},
\begin{equation*}
 \begin{split}
 &\sum_{|\boldsymbol{\alpha}|=n}
 |\mathfrak{P}_{\boldsymbol{\alpha}}(t,x,y)|^2
 \leq n!\int\limits_{(0,\pi)^n}
 \int\limits_{\bT^n_{0,t}}\int\limits_{\bT^n_{0,t}}
\Big(
\hker(t-s_n,x,y_n)\cdots \hker(s_{2}-s_1,y_{2},y_1)\hker(s_1,y_1,y)\\
&
\qquad
\times \hker(t-r_n,x,y_n)\cdots \hker(r_{2}-r_1,y_{2},y_1)\hker(r_1,y_1,y)
\Big)\,
 ds^n\, dr^n\  dy^n\\
 &\qquad
 \leq n! (1+\sqrt{t})^{2(n+1)} I_n^2\big(t;1/4,1/2\big)\leq \frac{C^n(t)}{n^{n/2}},
 \end{split}
 \end{equation*}
from which \eqref{FS-reg1} follows.

To establish \eqref{FS-sym-pos},
note that  \eqref{FCS-det} and Theorem \ref{th:positivity} imply $\mathfrak{P}\geq 0$,
whereas, by \eqref{mod-a-gen-tx}, using $\hker(t,x,y)=\hker(t,y,x)$ and
 a suitable change of the time variables in the integrals,
$$
\sum_{|\ba|=n}
 |\mathfrak{P}_{\boldsymbol{\alpha}}(t,x,y)-
 \mathfrak{P}_{\boldsymbol{\alpha}}(t,y,x)|^2=0, \ n\geq 1,
 $$
 which implies $\mathfrak{P}(t,x,y)=\mathfrak{P}(t,y,x)$.

To establish \eqref{FS-reg2} and \eqref{FS-reg3}, we compute,  for $n\geq 1$,
\begin{align*}
\sum_{|\boldsymbol{\alpha}|=n}&
 |\mathfrak{P}_{\boldsymbol{\alpha}}(t+h,x,y) -
 \mathfrak{P}_{\boldsymbol{\alpha}}(t,x,y)|^2
 \leq h^{2\gamma} C^n(t,\gamma)n^{-n/2},\ \gamma\in (0,3/4);
\\
&{\rm \ cf. \ \eqref{t-incr},\ \ and }\\
\sum_{|\boldsymbol{\alpha}|=n}&
 |\mathfrak{P}_{\boldsymbol{\alpha},x}(t,x+h,y)-
 \mathfrak{P}_{\boldsymbol{\alpha},x}(t,x,y)|^2
  \leq h^{2\gamma}C^n(t,\gamma)n^{-n/2},\ \gamma\in (0,1/2);\\
& {\rm \ cf. \ \eqref{sp-incr}. }
\end{align*}

Note that $\mathfrak{P}_{\zm}(t,x,y)=\hker(t,x,y)$ is infinitely
differentiable in $t$ and $x$ for $t>0$ but is unbounded as $t\searrow 0$;
cf. \eqref{eq:hker0}.

 \end{proof}

 Now we can give full justification of the reason why $\mathfrak{P}$
 is natural to call the fundamental chaos solution of equation \eqref{eq:main}.

 \begin{theorem}
 If  $u_0\in L_{2,q}\big(W;L_2((0,\pi))\big)$ for some $q>1$, then
the chaos solution of \eqref{eq:main} with initial condition
$u(0,x)=u_0(x)$ is
\bel{Gen-CS}
u(t,x)=\int_0^{\pi} \mathfrak{P}(t,x,y)\diamond u_{0}(y)dy,
\ee
and
\bel{Gen-CS-reg}
u(t,x) \in L_{2,p}(W;\bR)
\ee
for every $p<q$, $t>0$, and $x\in [0,\pi]$.
\end{theorem}

\begin{proof}
 Let
$$
u_0(x)=\sum_{\ba\in \cJ} u_{0,\ba}(x)\xi_{\ba}
$$
be the chaos expansion of the initial condition.
By definition, the chaos solution of \eqref{eq:main} is
$$
u(t,x)=\sum_{\ba\in \cJ} u_{\ba}(t,x)\xi_{\ba},
$$
where
\begin{equation*}
 \begin{split}
 \frac{\partial {u}_{\zm}(t,x)}{\partial t}&=
 \frac{\partial^2{u}_{\zm}(t,x)}{\partial x^2} ,\
u_{\zm}(0,x)=u_{0,\zm}(x);\\
\frac{\partial{u}_{\ba}(t,x)}{\partial t}&=
\frac{\partial^2 u_{\ba}(t,x)}{\partial x^2}
+\sum_k \sqrt{\alpha_k}\
\mfk{m}_k(x)u_{\ba^-(k)}(t,x),\\
& u_{\ba}(0,x)=u_{0,\ba}(x),\ |\ba|>0.
\end{split}
\end{equation*}
By \cite[Theorem 9.8]{LR_shir},
if $u(0,x)=f(x)\xi_{\boldsymbol{\beta}}$ for
some $f\in {L_2((0,\pi))}$ and ${\boldsymbol{\beta}}\in \cJ$, then
$$
u(t,x)=\int_0^{\pi} \mathfrak{P}(t,x,y)\diamond \xi_{\boldsymbol{\beta}}
 f(y)dy.
$$
Then \eqref{Gen-CS} follows by linearity.

Next, given $1\leq p<q$, take $p'=qp/(q-p)$, so that $p'^{-1}+q^{-1}
=p^{-1}$. Then, by \eqref{FS-reg1}  and  \cite[Theorem 4.3(a)]{LRS},
$$
\mathfrak{P}(t,x,\cdot)\diamond u_{0}\in L_{2,p}\big(W;L_2((0,\pi))\big),
$$
which implies \eqref{Gen-CS-reg}.

\end{proof}

\section{Further Directions}
\label{sec:FD}

A natural question is whether the results of
Sections \ref{sec:CSol}--\ref{sec:FS} extend
to a more  general equation
$$
\frac{\partial u(t,x)}{\partial t} =
\mathcal{L} u(t,x)+
u(t,x)\diamond \dot{W}(x),\ t>0,\  \ x\in G,
$$
where $\mathcal{L}$ is a second-order linear ordinary differential operator
and $G\subseteq \bR$.

\subsection{Equation on a Bounded Interval}

Consider a second-order differential \\
operator
$$
f\mapsto \mathcal{L}f=\rho(x)f''+r(x)f'+c(x)f = (\rho f')'+(r-\rho')f'+cf,\ \rho>0,\ x\in (a,b),
$$
$-\infty<a<b<+\infty$.

A change of variables $f(x)=g(x)\exp\left(
-\int\frac{r(x)-\rho'(x)}{2\rho(x)}dx\right)$
leads to the symmetric operator
$$
\tilde{\mathcal{L}}g= (\rho g')'+\tilde{c} g,
$$
with
$$
\tilde{c}(x)= c(x)+ \frac{\rho(x) H''(x)+r(x)H'(x)}{H(x)},\ \
H(x)=\exp\left(-\int\frac{r(x)-\rho'(x)}{2\rho(x)}dx\right).
$$

The most general form of the (real) homogenous
boundary conditions for the
operator $\tilde{\mathcal{L}}$ is as follows:
\bel{BC-gen}
AY(a)+BY(b)=0,
\ee
where $A,B\in \bR^{2\times 2}$, $Y(x)=\big(g(x) \ \ \rho(x)g'(x)
\big)^{\top}$. If
the matrix $[A\ B] \in \bR^{2\times 4}$ has rank 2 and
\bel{eq:SABC}
AEA^{\top}=BEB^{\top},\ \
E=
\left(
\begin{array}{cc}
0 & -1 \\
1 & 0
\end{array}
\right),
\ee
then boundary conditions \eqref{BC-gen} allow a self-adjoint
extension of the operator $\tilde{\mathcal{L}}$ to $L_2((a,b))$;
cf. \cite[Section 4.2]{Zettl}.
Particular cases of \eqref{BC-gen} satisfying \eqref{eq:SABC}
are {\tt separated boundary conditions}
$$
c_1g(a)+c_2\rho(a)g'(a)=0,\ c_3g(b)+c_4\rho(b)g'(b)=0,\
$$
when both matrix products in \eqref{eq:SABC} are zero,
and {\tt periodic boundary conditions}
$$
g(a)=g(b), \ \  g'(a)=g'(b),
 $$
when $A=B$.

Consider the eigenvalue problem for $\tilde{\mathcal{L}}$:
$$
-\tilde{\mathcal{L}}\mfk{m}_k(x)=\lambda^2_k\mfk{m}_k,\
k=1,2,\ldots.
$$
The following properties of the eigenvalues $\lambda_k$
and the  eigenfunctions $\mfk{m}_k$ ensure that the results
of   Sections \ref{sec:CSol}--\ref{sec:FS} extend to equation
$$
u_t=\tilde{\mathcal{L}}u + u\diamond \dot{W}(x).
$$
\begin{itemize}
 \item[{[EVS]}] The eigenvalues $\lambda_k$ satisfy
  $$
  \lim_{k\to \infty} \frac{\lambda_k}{k^2}=\bar{\lambda}>0;
  $$
\item[{[CEF]}]
The set of normalized eigenfunctions $\{\mfk{m}_k,\ k\geq 1\}$
is complete in $L_2((a,b))$;
\item[{[MP]}] The kernel
$$
\hker(t,x,y)=\sum_{k=1}^{\infty}
e^{-\lambda_k^2t}\;\mfk{m}_k(x)\mfk{m}_k(y)
$$
of the  semi-group generated by $\tilde{\mathcal{L}}$
 is non-negative: $\hker(t,x,y)\geq 0$, $t>0,\ x,y\in (a,b)$;
 \item[{[UB]}] The eigenfunctions $\mfk{m}_k$ are uniformly
 bounded:
 $$
 \sup_{k\geq 1,\ x\in (a,b)} |\mfk{m}_k(x)|<\infty.
 $$
\end{itemize}
The corresponding computations, although not necessarily trivial, are essentially
equivalent to what was done in Sections \ref{sec:CSol}--\ref{sec:FS}.
In the case of additive noise,
 the results of \cite{Feller-1D} help with identification of the
diffusion process that, similar to the Brownian bridge in Remark
\ref{rem:compensate},  compensates the spatial derivative of the solution
 to a smooth function.

There are general sufficient conditions for [EVS] and
[CEF] to hold; cf. \cite[Theorems 4.3.1 and 4.6.2]{Zettl}.
The maximum principle [MP] means special restrictions on the
boundary conditions; these restrictions are, in general, not related to
\eqref{BC-gen}; cf. \cite[Theorem 12]{Feller-1D}.
Condition [UB] appears to be the most difficult to verify without additional information about the operator $\tilde{\mathcal{L}}$.

\subsection{Equation on the Line}

All the results of  Sections \ref{sec:CSol}--\ref{sec:FS} extend to the
chaos solution of the heat  equation
$$
u_t=u_{xx}+u\diamond \dot{W}(x),\ t>0,\ x\in \bR,
 $$
 with suitable initial condition $u(0,x)=u_0(x)$. In the case of additive noise,
  a two-sided Brownian motion compensates the space derivative of the
  solution to a smooth function.

 Further extensions, to the equation
 $$
 u_t=\mathcal{L}u + u\diamond \dot{W}(x),
 $$
 are also possible but might require additional effort.

 For example, let $\mathcal{L}f=\big(c(x)f'(x)\big)'$
  with a measurable function $c(x)$  such that
$$
c_1 \leq c(x) \leq c_2,\quad \mbox{ for all } x\in \mathbb{R},
$$
for some constants $c_1, c_2 >0$.
Then \cite[Theorems 4.1.11 and 4.2.9]{Stroock}
the kernel of the the corresponding semi-group satisfies
$$
 \frac{\mfk{c}_1}{\sqrt{t}}\exp
 \left(-\frac{(x-y)^2}{\mfk{a}_1\,t}\right)
 \leq \hker(t,x,y)=\hker(t,y,x)
  \leq \frac{\mfk{c}_2}{\sqrt{t}}\exp
 \left(-\frac{(x-y)^2}{\mfk{a}_2\,t}\right)
$$
for some positive numbers $\mfk{c}_1,\mfk{c}_2,\mfk{a}_1,\mfk{a}_2$,
which is enough to carry out the computations from Sections \ref{sec:CSol}
and \ref{sec:RCS}.
Additional regularity of the chaos solution requires more delicate bounds on
$\hker_x(t,x,y)$ and $\hker_t(t,x,y)$.


\def\cprime{$'$}
\providecommand{\bysame}{\leavevmode\hbox to3em{\hrulefill}\thinspace}
\providecommand{\MR}{\relax\ifhmode\unskip\space\fi MR }
\providecommand{\MRhref}[2]{%
  \href{http://www.ams.org/mathscinet-getitem?mr=#1}{#2}
}
\providecommand{\href}[2]{#2}

\end{document}